\documentclass[a4paper,12pt]{article}

\usepackage{amsthm,amsmath,stmaryrd,bbm,hyperref,geometry,color,authblk}
\usepackage[utf8]{inputenc}
\usepackage{amssymb}
\usepackage[english]{babel}
\usepackage{graphicx}
\usepackage{amsfonts,amssymb}
\usepackage{verbatim}
\usepackage{enumitem}
\usepackage[mathscr]{euscript}

\newcommand{\po}{\left(}
\newcommand{\pf}{\right)}
\newcommand{\co}{\left[}
\newcommand{\cf}{\right]}

\newcommand{\R}{\mathbb R} 
\newcommand{\T}{\mathbb T} 
 
\newcommand{\N}{\mathbb N}

\newcommand{\dd}{\text{d}}
\newcommand{\FF}{\mathcal F}
\newcommand{\bX}{\mathbf{X}}
\newcommand{\bV}{\mathbf{V}}
\newcommand{\bv}{\mathbf{v}}
\newcommand{\bx}{\mathbf{x}}

\newcommand{\bY}{\mathbf{Y}}

\newcommand{\COT}{\mathscr{C}_{OT}}
\newcommand{\Heps}{\mathscr{H}_\varepsilon}

\newcommand{\na}{\nabla}
\newcommand{\1}{\mathbbm{1}}

\newtheorem{theorem}{Theorem}
\newtheorem{assumption}{Assumption}
\newtheorem{lemma}[theorem]{Lemma}

\newtheorem{proposition}[theorem]{Proposition}
\newtheorem{remark}{Remark}
\newtheorem{example}{Example}

\title{Nesterov acceleration for the Wasserstein  minimization of displacement-convex free energies}
\author{Pierre Monmarché}

\begin{document}
\maketitle

\begin{abstract}
We show that the mean-field underdamped Langevin process (associated to the non-linear Vlasov-Fokker-Planck equation) achieves a Nesterov acceleration with respect to the Wasserstein gradient flow of a displacement-convex free energy, in the sense that it converges at a rate of order given by the square-root of the Polyak-Łojasiewicz constant of the free energy (which is the optimal convergence rate for the corresponding gradient flow). This result has been made possible by the recent breakthrough \cite{Lu} by Jianfeng Lu, which establishes  such a \emph{diffusive-to-ballistic} improvement in terms of entropy in the linear case.
\end{abstract}

\section{Introduction and result}

\paragraph{Momentum and acceleration.} The term \emph{Nesterov acceleration}, named after \cite{nesterov1983method}, refers to the situation where the convergence rate of a convex optimization scheme which uses some momentum/inertia is quadratically improved with respect to the standard gradient descent. This phenomenon had already been noticed by Polyak in \cite{polyak1964some}. Both Nesterov's and Polyak schemes can be interpreted as discretizations of the damped Hamiltonian (a.k.a. heavy ball) dynamics
\begin{equation}
\label{loc:heavyball}
\left\{\begin{array}{rcl}
\dot x &= & v \\
\dot v & =& - \na f(x) - \gamma v\,,
\end{array}\right.
\end{equation}
with $f\in\mathcal C^1(\R^d,\R)$ the objective function and $\gamma>0$ a friction parameter, see \cite{shi2022understanding,su2016differential}. At the continuous-time level, the acceleration can be described as the fact that the long-time convergence rate of the solution of~\eqref{loc:heavyball} to $\mathrm{argmin} f$ is of order $\sqrt{\lambda}$ where $\lambda$ is the optimal convergence rate of the gradient descent
\[\dot x = -\na f(x)\,.\]
This is for instance simple to see when $f(x) = |A x - b|^2$ for some matrix $A$ and vector $b$, in which case the problem boils down to an eigenvalue computation. If we normalize the maximal eigenvalue of $A^T A$ to be one (which amounts to normalize time, having in mind that, with a time discretization, the step-size will be constrained by this largest eigenvalue), $\lambda$ is given by the smallest eigenvalue of $AA^T$, hence its condition number. When the matrix is badly conditioned (i.e. the problem is highly anisotropic/multiscale), $\lambda$ is very small, and thus $\sqrt{\lambda} \gg \lambda$.

 This continuous-time statement can then usually be transfered to the discrete-time practical algorithms, up to a discretization error analysis and possibly a time change. This is still an active topic in optimization; we refer the interested reader to e.g.  \cite{shi2022understanding,su2016differential,wilson2021lyapunov,even2021continuized} and references within for a relatively recent panorama. In the present work we will only discuss the continuous-time problem.
 
 \paragraph{PL inequalities.} Assume that $f$ is normalized so that $\inf f= 0$. In order to establish an acceleration, we have to take the \emph{optimal} convergence rate of the gradient flow as reference. A convenient notion of optimality here is to consider the largest $\lambda>0$ such that 
 \begin{equation}
 \label{loc:PLft}
 \forall x_0  \in \R^d,t\geqslant 0,\qquad f(x_t) \leqslant e^{-2\lambda t} f(x_0) ,
 \end{equation}
 which by expanding this inequality at $t=0$ or by differentiating $f(x_t)$ over time can be seen to be equivalent to
 \begin{equation}
 \label{eq:PLf}
 \forall x\in\R^d,\qquad f(x)  \leqslant \frac{1}{2\lambda}|\na f(x)|^2\,.
 \end{equation}
 This is called a Polyak-Łojasiewicz (PL) inequality with constant $\lambda$. It clearly implies that all critical points of $f$ are global minimizers. Assume that $f$ admits a unique minimizer $x_*$. It is well-known that~\eqref{eq:PLf} then implies that
 \begin{equation}
 \label{loc:PLdivfree}
 \forall x\in\R^d,\qquad |x-x_*|^2 \leqslant \frac{2}{\lambda} f(x), 
 \end{equation}
 see e.g. \cite[Equation (6)]{monmarche2025local}, which is another form of Łojasiewicz inequality~\cite{law1965ensembles} (sometimes referred to as a quadratic growth condition~\cite{karimi2016linear}). Notice that, combined with~\eqref{loc:PLft}, it shows that $\lambda$ is indeed a convergence rate for $|x_t-x_*|$, since
 \[|x_t-x_*|\leqslant \sqrt{2f(x_t)/\lambda} \leqslant \sqrt{2f(x_0)/\lambda} e^{-\lambda t}\,.\] 
 In fact, when $f$ is convex,
 \begin{equation}
 \label{loc:fxx*}
 f(x) \leqslant |x-x_*||\na f(x)|\,, 
 \end{equation}
 which means that the derivative-free inequality~\eqref{loc:PLdivfree} implies the PL one~\eqref{eq:PLf} with constant $\lambda/4$ instead of $\lambda$. If $f$ is even strongly convex,~\eqref{loc:fxx*} can be improved to \cite[Equation (1.11)]{CMCV} and the constant deduced for~\eqref{eq:PLf} is improved, the ideal case being the quadratic situation $f(x) = \frac{c}{2}|x|^2$ where~\eqref{loc:PLdivfree} and~\eqref{eq:PLf} are perfectly equivalent with the same optimal constant $\lambda=c$.
 
 In the present work, we will use as a reference the optimal constant $\lambda$ such that~\eqref{loc:PLdivfree} holds, and we will refer to it as the PL constant of $f$. Since we are interested in the convex case, it is in $[\lambda',4\lambda']$ with $\lambda'$ the optimal constant for~\eqref{eq:PLf} (which is perhaps more classically considered when analyzing gradient flows).

\paragraph{Wasserstein optimization.} The discussion above was presented in $\R^d$ for clarity, however in this work we consider optimization problems over the Wasserstein space $\mathcal P_2(\R^d)$ of probability measures with finite second moment. Given $\mathcal E:\mathcal P_2(\R^d) \rightarrow (-\infty,\infty]$, referred to as an energy, we are interested in the question of minimizing the entropic-regularized objective
\begin{equation}
\label{loc:freeenergy}
\mathcal F(\rho) = \mathcal E(\rho) + \mathcal H(\rho)\,,
\end{equation}
referred to as a free energy, where $\mathcal H(\rho) $ is Boltzmann's entropy, equal to $ \int_{\R^d} \rho \ln \rho$ if $\rho$ has a Lebesgue density (also denoted $\rho$) and $+\infty$ otherwise.

Such optimization problems arise in a number of situations. Some important models of interacting particles in statistical physics are known to converge to minimizers of such free energies as the size of the population goes to infinity, as in \cite{bashiri2020gradient,sandier2004gamma,bunne2022proximal}, the entropy term arising from thermal fluctuations or other stochastic effects at the microscopic level (in \eqref{loc:freeenergy}, the temperature has been scaled to $1$). In order to explain their performances, many high-dimensional algorithms have been shown to converge as the dimension goes to infinity to the gradient flow of a suitable functional over $\mathcal P_2(\R^d)$ \cite{Szpruch,mei2018mean,geshkovski2025mathematical}. Alternatively, in variational inference, the objectives are directly stated in term of an optimization problem in a space of probability measures \cite{arbel2019maximum,lambert2022variational}. For optimization algorithms (either directly in $\mathcal P_2(\R^d)$ or arising as a high dimensional limit), the entropy term may either be a penalization added to ensure some regularity and that minimizers have a density~\cite{chizat,peyre2015entropic}, or it can be a toy model for the noise arising from the use of stochastic approximations in the practical algorithms \cite{menon2026implicit}. Finally, the linear case where $\mathcal E(\rho) = \int_{\R^d} V \rho$ for some potential $V:\R^d\rightarrow \R$ corresponds to the problem of sampling the Gibbs measure proportional to $e^{-V}$ (since it is the unique minimizer of $\mathcal F(\rho)$), and some enhanced sampling method can also be interpreted as optimization schemes of a free energy associated to some non-linear energy \cite{lelievre2026convergence}.

Since we are interested in establishing a Nesterov acceleration, we are concerned with convex cases. There are several notions of convexity in $\mathcal P_2(\R^d)$. The one involved in our analysis is that of displacement-convexity:  $\mathcal E$ is said to be displacement-convex if it is convex along $\mathcal W_2$-geodic curves, i.e. $t\mapsto \mathcal E(\rho_t)$ is convex whenever $t\mapsto \rho_t$ is a geodesic for the $\mathcal W_2$ distance~\eqref{def:W2}. For instance, the entropy $\mathcal H$ is displacement-convex, and $\mathcal E(\rho)=\int_{\R^d} V \rho$ is displacement convex if and only if $V$ is convex. See \cite{ambrosio2005gradient,CMCV} and Example~\ref{exemple} for these facts and further details.

We consider exactly the same conditions in this Wasserstein optimization problem as in the finite-dimensional situation discussed above, namely:

\begin{assumption}\label{assum:Convexetal}
The energy $\mathcal E$ is such that : 
\begin{enumerate}
\item $\mathcal E $ is displacement-convex. 
\item The free energy $\mathcal F$ admits a global minimizer $\rho_*$, which is unique. Without loss of generality, $\FF(\rho_*) = 0$.
\item The PL constant  
\begin{equation}
\label{eq:lambdaPL}
\lambda_* := \inf\left\{\frac{2\FF(\rho)}{\mathcal W_2^2(\rho,\rho_*)},\ \rho\in\mathcal P_2(\R^d),\ \rho\neq \rho_*\right\}
\end{equation}
is positive.
\end{enumerate}
\end{assumption}

In particular, under Assumption~\ref{assum:Convexetal}, $\FF$ is also displacement-convex. The analogue of the PL inequality~\eqref{eq:PLf} reads
\begin{equation}\label{loc:PLW2}
\forall \rho\in\mathcal P_2(\rho)\,,\qquad \FF(\rho) \leqslant \frac{1}{2\lambda} \mathcal I(\rho),\qquad \mathcal I(\rho) = \int_{\R^d} \left|\na \ln \rho + D\mathcal E(\rho,\cdot)\right|^2 \rho, 
\end{equation}
where $D\mathcal E$ the intrinsic derivative of $\mathcal E$ is defined in~\eqref{def:DE}, and it implies that $\lambda_*\geqslant \lambda$. In general non-displacement-convex cases, this stronger form of PL inequality is more classically considered, see e.g. \cite{chizat,monmarche2025local,Pavliotis}, and it is the optimal constant $\lambda$ such that 
\begin{equation}
\label{loc:decayGrad}
\forall m_0 \in \mathcal P_2(\R^d),\ t\geqslant 0,\qquad \FF(m_t) \leqslant e^{-2\lambda t} \FF(m_0)
\end{equation}
 along the gradient flow (see~\eqref{eq:GradientFlow} below), cf. \cite{Pavliotis,monmarche2025local}.  Moreover, like~\eqref{eq:PLf}, it implies that all critical points of $\FF$ (i.e. solutions of $\mathcal I(\rho)=0$) are global minimizers. Notice that, by the analogue in this situation of~\eqref{loc:fxx*}, which is  called the HWI inequality (see e.g.~\cite[Equation (1.15)]{CMCV}), under Assumption~\ref{assum:Convexetal},
\[\forall \rho \in \mathcal P_2(\R^d),\qquad \FF(\rho) \leqslant \mathcal W_2(\rho,\rho_*) \sqrt{\mathcal I(\rho)}\]
and thus~\eqref{loc:PLW2}  holds with optimal constant $\lambda\in[ \lambda_*/4,\lambda_*]$. In particular, under Assumption~\ref{assum:Convexetal}, $\rho_*$ is the only critical point of $\mathcal F$.

Apart from Assumption~\ref{assum:Convexetal}, we will also require some regularity conditions, stated as Assumption~\ref{assum:regularite} below.

\paragraph{Gradient flow and kinetic counterpart.} The gradient flow associated to $\FF$, defined within the general theory of gradient flows in metric spaces (see \cite{ambrosio2005gradient} for a general exposure) is given by the parabolic equation
\begin{equation}
\label{eq:GradientFlow}
\partial_t m_t = \na \cdot \po D\mathcal E(m_t,\cdot) m_t \pf + \Delta m_t\,.
\end{equation}
 Its stationary solutions are exactly the critical points of $\FF$ and thus, under Assumption~\ref{assum:Convexetal}, $\rho_*$ is the only one. As discussed above, the largest constant $\lambda$ possible for the decay~\eqref{loc:decayGrad} is smaller than~$\lambda_*$ in~\eqref{eq:lambdaPL}. As a consequence, a Nesterov acceleration is obtained if a process with momentum converges at a rate of order $\sqrt{\lambda_*}$.

We consider the  natural kinetic analogue of~\eqref{eq:GradientFlow}, which is the Vlasov-Fokker-Planck equation over $\mathcal P(\R^d\times\R^d)$:
\begin{equation}
\label{eq:VFP}
\partial_t \nu_t + v\cdot \na_x \nu_t - D\mathcal E(\rho_t,x) \cdot \na_v \nu_t =   \gamma \co \na_v\cdot\po      v  \nu_t \pf +   \Delta_v \nu_t\cf \,,\qquad \rho_t := \int_{\R^d} \nu_t(\cdot,v)\dd v\,,
\end{equation}
with $\gamma>0$ called the friction parameter. The solution $\nu_t(x,v)$ is interpreted as a density of particles at position $x$ and velocity $v$. As $\gamma\rightarrow \infty$ (overdamped regime), $\rho_{\gamma t}$ converges to a solution of~\eqref{eq:GradientFlow}, see e.g. \cite{lelievre2016partial}.  A measure $\hat \nu_* \in \mathcal P(\R^{2d})$ is a stationary solution of~\eqref{eq:VFP} if and only if $\hat \nu_* =\hat \rho_* \otimes \kappa$ with $\hat \rho_*$ a stationary solution of~\eqref{eq:GradientFlow} and $\kappa(v) \propto e^{-\frac12|v|^2}$ the standard Gaussian distribution over $\R^d$. This justifies the use of~\eqref{eq:VFP} in order to minimize $\FF$. In particular, under Assumption~\ref{assum:Convexetal}, the unique stationary solution of~\eqref{eq:VFP} is 
\begin{equation}
\label{eq:nu*}
\nu_* := \rho_* \otimes \kappa \,.
\end{equation}
We introduce the kinetic free energy
\begin{equation}
\label{eq:Fk}
\end{equation}
\[\mathcal F_k(\nu) = \mathcal E(\rho)  + \frac12\int_{\R^d} |v|^2\nu(x,v)\dd x\dd v + \mathcal H(\nu)  + \frac{d}{2}\ln(2\pi)\,.  \]
Then $t\mapsto \FF_k(\nu_t)$ is non-increasing (see e.g. \cite{monmarche2025local}) and, under Assumption~\ref{assum:Convexetal}, $\nu_*$ is the global minimizer of $\FF_k$ (which has been normalized so that $\FF_k(\nu_*)=0$).

The fact that~\eqref{eq:VFP} involves some momentum by contrast to~\eqref{eq:GradientFlow} can be understood with a Lagrangian description of these equations. Indeed,~\eqref{eq:GradientFlow} can be interpreted as the Kolmogorov equation satisfied by the law of the (non-linear) overdamped Langevin process solving
\begin{equation}\label{loc:ZtMcK}
\dd Z_t = -D\mathcal E(m_t,Z_t) \dd t + \sqrt{2}\dd B_t,\qquad m_t =  Law(Z_t)\,,
\end{equation}
where $B$ is a Brownian motion, while~\eqref{eq:VFP} corresponds to a (non-linear, kinetic) Langevin diffusion process
\begin{equation}\label{loc:XVMcK}
\left\{\begin{array}{rcl}
\dd X_t &= & V_t\dd t \\
\dd V_t &=& -D\mathcal E(\rho_t,X_t) \dd t - \gamma V_t \dd t + \sqrt{2\gamma}\dd B_t \qquad \rho_t =  Law(X_t)\,.
\end{array}\right.
\end{equation}

 Momentum-based accelerated flows over the Wasserstein space have been considered previously~\cite{carrillo2019convergence,JMLR:v26:23-1288,wang2022accelerated}. However, these works do not  specifically  consider objective functions with an entropic part, and if we apply their algorithms with objective function $\mathcal F$ we do not get~\eqref{eq:VFP}. For instance, the Wasserstein version of the heavy ball algorithm, studied in~\cite{JMLR:v26:23-1288}, reads
 \begin{equation}
 \label{loc:HBJMLR}
\partial_t \tilde \nu_t + v\cdot \na_x \tilde  \nu_t - D\mathcal F(\tilde \rho_t,x) \cdot \na_v \tilde \nu_t =  \gamma   \na_v\cdot\po      v  \tilde  \nu_t \pf \,, \qquad \tilde \rho_t := \int_{\R^d} \tilde \nu_t(\cdot,v)\dd v\,.
 \end{equation}
 Contrary to~\eqref{loc:XVMcK}, there is no simple stochastic interpretation of this equation. In practice, to implement a particle approximation of this, the term $D\mathcal H(\tilde \rho_t,x) = \na_x \ln \tilde \rho_t(x)$ has to be approximated (e.g. with kernel estimation or by a neural network as in score-based diffusion algorithms).

More generally, due to the  term $\Delta_v \nu_t$, it doesn't seem that~\eqref{eq:VFP} can be seen as a Hamiltonian flow over probability measures (in the sense of \cite[Definition 4]{JMLR:v26:23-1288}) for a suitable time-dependent Hamiltonian.  Integrating~\eqref{eq:VFP} over $v$ leads to the continuity equation
\begin{equation}
\label{loc:continuitEQ}
\partial_t \rho_t + \na\cdot \po \rho_t w_t \pf = 0
\end{equation}
with the velocity field $w_t$ given by the conditional average of the velocities:
\[w_t(x) = \frac{1}{\rho_t(x)}\int_{\R^d} v \nu_t(x,v)\dd v\,.\] 
Within Otto's  calculus over $\mathcal P_2(\R^d)$ \cite{otto2001geometry}, \eqref{loc:continuitEQ} can be understood as the first line of~\eqref{loc:heavyball}, but the interpretation of an analogue of the second line is less clear than the situations in~\cite{carrillo2019convergence,JMLR:v26:23-1288,wang2022accelerated}.

However, there is an alternative way to see~\eqref{eq:VFP} as a Wasserstein analogue to~\eqref{loc:heavyball}, by interpreting the latter as the juxtaposition of the Hamiltonian dynamics and, playing the role of a damping, the (partial) gradient descent (with learning rate $\gamma$)
\begin{equation*}
\left\{\begin{array}{rcl}
\dot x &= & \na_v H(x,v) \\
\dot v & =& -\na_x H(x,v)\,,
\end{array}\right.\qquad \dot v = - \gamma \na_v H(x,v)
\end{equation*}
with Hamiltonian $H(x,v) = f(x)+\frac{1}{2}|v|^2$, exactly as~\eqref{eq:VFP} is the juxtaposition of the Hamiltonian flow (as in \cite[Definition 4]{JMLR:v26:23-1288})
\begin{equation}
\label{loc:Hamilt}
\partial_t \nu_t + \na_x \cdot \po \nu_t \na_v \frac{\delta \mathcal F_k}{\delta m}(\nu_t,\cdot) \pf - \na_v \cdot \po \nu_t \na_x \frac{\delta \mathcal F_k}{\delta m}(\nu_t,\cdot) \pf =0
\end{equation}
and partial gradient flow (as in~\eqref{eq:GradientFlow})
\begin{equation}
\label{loc:damping}
\partial_t \nu_t =  \gamma \na_v \cdot \po \nu_t \na_v \frac{\delta \mathcal F_k}{\delta m}(\nu_t,\cdot) \pf\,,
\end{equation}
with both equations having for Hamiltonian the kinetic free energy~\eqref{eq:Fk}. Notice that, with this interpretation, we could juxtapose the Hamiltonian flow~\eqref{loc:Hamilt} with a (partial) gradient flow (in $v$) of $\mathcal F_k$ associated to another metric than $\mathcal W_2$, such as Fisher-Rao, Kalman-Wasserstein or Stein metrics (see \cite{wang2022accelerated} and references within). In the presence of an entropy (or similar but more general internal energy, see~\cite{ambrosio2005gradient}), this leads to kinetic particles with the same stochastic mechanism as their overdamped counterparts (such as the Brownian noise in~\eqref{loc:ZtMcK} and~\eqref{loc:XVMcK}, birth-and-death processes for the Fisher-Rao metric, etc.), which differ from the processes considered e.g. in \cite{wang2022accelerated}  (exactly as the heavy ball~\eqref{loc:HBJMLR} differs from the Vlasov-Fokker-Planck equation~\eqref{eq:VFP}). A way to see the difference between the point of view explained here and~\cite{wang2022accelerated} is that, in the extended Hamiltonian $\FF_k$, with respect to the initial objective $\FF$, we did not only add the kinetic energy, but also took the whole entropy of the joint density, from which the entropy terms vanish in the Hamiltonian flow~\eqref{loc:Hamilt} and appear in the damping part~\eqref{loc:damping}.

\medskip

Comparing the convergence rates of the two continuous-time equations~\eqref{eq:GradientFlow} and~\eqref{eq:VFP} hides an arbitrary choice of time normalization. Indeed, we could accelerate time by an arbitrary factor and completely change the comparison of the convergence rates. In fact, eventually, comparing the algorithms  makes sense without ambiguity only when specific discrete-time numerical schemes are chosen. However, comparing~\eqref{eq:GradientFlow} and~\eqref{eq:VFP} at the continuous-time level is reasonable because the term $D\mathcal E(\rho_t,X_t)$ appears with the same coefficient in both stochastic differential equations, which means that the impact of its Lipschitz constant on the choice of the discretization step should be the same in both cases. Moreover, we will see that, to get an acceleration, we should take $\gamma$ of order $\sqrt{\lambda_*}$, which means that the friction/dissipation part of~\eqref{loc:XVMcK} should not have an arbitrarily bad effect on time discretization when $\lambda_*$ is small (which is the regime we are interested in) and that the trace of the diffusion matrices (i.e. the ``amount of randomness per unit time") is at most of the same order in~\eqref{loc:XVMcK} as in~\eqref{loc:ZtMcK}. See also Remark~\ref{rem:discrete} about time discretization.

\paragraph{Main result and strategy.}  We can now state the main result of this work. Under Assumption~\ref{assum:regularite}, the Vlasov-Fokker-Planck equation~\eqref{eq:VFP} has a solution for any $\nu_0 \in \mathcal P_2(\R^{2d})$.  
\begin{theorem}
\label{thm:main}
Let $\Gamma>0$. Under Assumptions~\ref{assum:Convexetal} and~\ref{assum:regularite}, for any initial condition $\nu_0 \in \mathcal P_2(\R^{2d})$, a solution of~\eqref{eq:VFP} with friction $\gamma = \Gamma\sqrt{\lambda_*} $ satisfies
\begin{equation}
\label{eq:thm_main}
\forall t\geqslant 0,\qquad  \FF_k(\nu_t) \leqslant \frac{1+\theta}{1-\theta} \exp \po -  \frac{\theta \sqrt{\lambda_*}}{2(1+\theta)} t  \pf \FF_k(\nu_0)\,,
\end{equation}
for any $0<\theta \leqslant \min\{\frac{\Gamma}{12},\frac1{4\Gamma}\}$.
\end{theorem}
In practice, the velocity can be sampled at equilibrium, meaning that $\nu_0 = \rho_0 \otimes \kappa$ for some $\rho_0 \in \mathcal P_2(\R^d)$, in which case $\FF_k(\nu_0)=\FF(\rho_0)$. In this situation, since moreover $\mathcal F_k(\nu_t)\geqslant \mathcal F(\rho_t)$ by subadditivity of the relative entropy (see e.g. \cite[Lemma 18]{GuillinWuZhang}), we get 
\[\forall t\geqslant 0,\qquad  \FF(\rho_t) \leqslant \frac{1+\theta}{1-\theta} \exp \po -  \frac{\theta \sqrt{\lambda_*}}{2(1+\theta)} t  \pf \FF(\rho_0)\,.
\]
In particular, a Nesterov acceleration is achieved with respect to~\eqref{loc:decayGrad}. An accelerated rate is deduced for $\mathcal W_2(\rho_t,\rho_*) \leqslant \sqrt{2\mathcal F(\rho_t)/\lambda_*}$.

In the linear case where $\mathcal E(\rho) = \int_{\R^d} V\rho$, Theorem~\ref{thm:main} boils down to the main result of~\cite{Lu}. In other words, our work is an extension of Jianfeng Lu's work to non-linear settings. To go from the linear settings to the non-linear ones, we use the same method as in \cite{MONMARCHE20171721,guillin2021uniform}, which is to apply the linear result to a system of $N$ mean-field interacting particles and obtain the result by propagation of chaos as $N\rightarrow \infty$ (as in \cite{Pavliotis,GuillinWuZhang} for the gradient flow~\eqref{eq:GradientFlow}). The works \cite{MONMARCHE20171721,guillin2021uniform}, which are not restricted to convex cases, are based on Villani's modified entropy method for the linear Langevin process~\cite{Villani}. As such, they do not provide sufficiently sharp convergence rate in the convex case to establish an acceleration.

Even when applying the result of Lu to the $N$ particle system, there remains one difficulty, which is that we want the convergence rate in~\eqref{eq:thm_main} to be given in terms of the PL constant $\lambda_*$ in~\eqref{eq:lambdaPL}. However, applying the result of~\cite{Lu} to the $N$ particle system and letting $N\rightarrow \infty$ yields~\eqref{eq:thm_main} with $\lambda_*$ replaced by
\[\tilde \lambda_* = \limsup_{N\rightarrow \infty} \lambda_N \]
where $\lambda_N$ is the log-Sobolev constant of the $N$-particle Gibbs measure, see~\eqref{eq:LSI} and \eqref{eq:GIbbs}. It is known that $\tilde \lambda_* \leqslant \lambda_*$, see \cite{Pavliotis,GuillinWuZhang}. The equality $\tilde \lambda_* = \lambda_*$ is conjectured in \cite{Pavliotis}, but as of today it is not even known whether $\tilde \lambda_*>0$ in general if $\lambda_*>0$. See \cite{Dagallier,SongboLSI,MonmarcheLSI,wang2025large,MonmarcheToyModelMF} for some recent developments on this question.

We are not going to prove the conjecture of~\cite{Pavliotis}. Instead, rather than applying directly the result of~\cite{Lu}, we will follow the proof and, whenever the log-Sobolev inequality is applied, exploiting the fact that it is not applied to an arbitrary $\nu \in \mathcal P_2(\R^{2dN})$ but to the law at time $t$ of the particle system with initial condition $\nu_0^{\otimes N}$ for which propagation of chaos holds, we will replace it by an approximate inequality. In fact, in \cite{Lu}, the log-Sobolev inequality is only used through the Talagrand inequality it implies (corresponding to using~\eqref{loc:PLdivfree} deduced from~\eqref{eq:PLf}). In Proposition~\ref{prop:approxTalagrand}, we prove an approximate Talagrand inequality for the $N$-particle Gibbs measure which involves the mean-field PL constant $\lambda_*$ from~\eqref{eq:lambdaPL} and additional error terms that will eventually vanish as $N\rightarrow \infty$.

Let us notice that, after \cite{MONMARCHE20171721,guillin2021uniform}, it was shown in \cite{Songbo} that Villani's modified entropy method could be implemented directly at the level of the mean-field non-linear equation, without going through the $N$ particle approximation. This allows to exploit directly the non-linear PL inequality instead of assuming the (a priori stronger) uniform-in-$N$ log-Sobolev inequality. 
It is probably possible to do this also for the arguments in~\cite{Lu}, which would give an alternative (arguably more self-contained) proof of Theorem~\ref{thm:main}. 

\begin{remark}\label{rem:discrete}
In the settings of Theorem~\ref{thm:main}, by the PL inequality and the Wasserstein-to-free energy regularization result from \cite[Theorem 3(12)]{guillin2021uniform} (see also Proposition 19 combined with Equation~(79) in \cite{monmarche2025local} for general settings, \cite{guillin2021uniform} being only stated for pairwise interaction energies),~\eqref{eq:thm_main} implies
\begin{equation}\label{loc:congract}
\forall t\geqslant 0,\qquad  \mathcal W_2(\nu_t,\nu_*) \leqslant C' \exp \po -  \frac{\theta \sqrt{\lambda_*}}{2(1+\theta)} t  \pf \mathcal W_2(\nu_0,\nu_*)
\end{equation}
for some explicit $C'>0$. The interest of this with respect to~\eqref{eq:thm_main} is that, by triangular inequality, we can bound
\[\mathcal W_2(\tilde \nu_{t+s} ,\nu_*) \leqslant \mathcal W_2(\tilde \nu_{t+s},\Phi_s(\nu_t)) + \mathcal W_2(\Phi_s(\nu_t),\nu_*)\,,\]
with $\Phi_s(\nu)$ the solution of~\eqref{eq:VFP} at time $s$ initialized at $\nu$ and $\tilde \nu_t$ a numerical scheme (e.g. the empirical distribution of a time discretization of an interacting particle system). It then suffices to combine~\eqref{loc:congract} with a finite-time error analysis bounding $\mathcal W_2(\tilde \nu_{t+s},\Phi_s(\nu_t))$ uniformly over $t\geqslant 0$ and $s\in[0,T]$ for some $T>0$ to get a uniform bound on $\mathcal W_2(\tilde \nu_{t+s} ,\nu_*) $ with a contraction rate of the same order as~\eqref{loc:congract}, as in \cite{durmus2024asymptotic,10.3150/19-BEJ1178}. See \cite{altschuler2025shifted} for a more recent and finer development of this strategy.

\end{remark}

\paragraph{Previous works on the linear case.} To conclude this introduction, let us discuss the literature which has been  concerned with kinetic acceleration for sampling, corresponding to the linear case where $\mathcal E(\rho)=\int_{\R^d}V\rho$, and in particular  the recent developments of this active field. A seminal work on this topic is the analysis of Diaconis, Holmes and Neal in~\cite{diaconis2000analysis} of a persistent walk to sample the uniform measure over the discrete torus $\T_N=\mathbb Z/N\mathbb Z$. It provided a mathematical evidence of a so-called \emph{diffusive to ballistic} acceleration (which had already been noticed in physics): in $k$ steps, the symmetric random walk over $\T_N$ typically moves at a distance of order $\sqrt{k}$ (as long as $k=\mathcal O(N)$) from its starting point. This is referred to as a diffusive behavior, and it means that $N^2$ iterations are required to cover the typical distances in $\T_N$, which are of order $N$, causing the mixing time to be order $N^2$. By contrast, making the velocity of the persistent walk switching between $-1$ and $1$ at a rate of order $1/N$, the chain follows its initial velocities for times of order $N$, hence on average moves at a distance of order $k$ with $k$ iterations (corresponding to a so-called ballistic behavior) as long as $k=\mathcal O(N)$. In other words, the inertia lasts over a time-scale sufficient to cover typical distances over $\T_N$ with a ballistic behavior, eventually yielding a mixing time of order $N$.

Taking scaling limits  (as $N\rightarrow \infty$)  of variations of the persistent walks led to the introduction of piecewise-deterministic velocity jump samplers \cite{PetersdeWith,M24,BierkensRoberts}, which have since then been the topic of many developments, see  e.g. \cite{10.1214/20-AAP1659,M26,andrieu2021hypocoercivity} for further references and some convergence rates in general situations. Apart from toy models in \cite{gadat2013spectral,Monmarche2013}, the first convergence rates obtained for this class of kinetic processes were not sharp enough to describe a diffusive to ballistic acceleration.

The same goes for the Langevin diffusion~\eqref{loc:XVMcK} (with $D\mathcal E(\rho_t,\cdot)=\na V$ in the linear case) and related generalized Hamiltonian Monte Carlo methods \cite{idealized,M41} and associated discrete-time feasible schemes, where the proofs of an acceleration first remained restricted to Gaussian targets, for which basically everything boils down to linear algebra (see the rich bibliography in \cite[Section 1.3]{altschuler2025shifted}). In particular, in the Gaussian case, the interpretation of the acceleration in terms of ballistic behavior lasting up to the time-scale of typical distances, as in the persistent walk, is still clear, see the last discussion in \cite{M41}.

An important progress was made in the general convex settings when the  space-time Poincaré inequality method was introduced in~\cite{albritton2024variational}. Indeed, this launched a series of work \cite{cao2023explicit,lu2022explicit,EberleLift,Brigati} based on this method where sharp convergence rates were obtained for kinetic processes  (including lower bounds showing that the acceleration cannot be better than quadratic for kinetic processes, see \cite{EberleLift}). These rates are proven for continuous-time processes and are stated in terms of $L^2$ norm (a.k.a. $\chi^2$ divergence). Very recently, it was shown in~\cite{fan2026sharp} that these sharp rates can in fact also be obtained with the earlier  Dolbeault-Mouhot-Schmeiser (DMS) modified $L^2$ norm approach~\cite{dolbeault2015hypocoercivity}. However, in the mean-field situation, by contrast with the relative entropy, the $L^2$ norm does not scale well with the number of particles, so that our approach to establish Theorem~\ref{thm:main} does not apply. Even when working directly with the non-linear equation, the $L^2$ approach, which can be seen as a linearization of the entropic situation for perturbations of the stationary solution, only applies to small non-linear perturbations of the linear equations or small perturbations of a stationary solution, as in \cite{herau2007short,bouin2026quantitative}.

Contrary to the $L^2$ case, before \cite{Lu}, the only available approach to get entropic convergence rates for the underdamped Langevin process had been Villani's modified entropy method~\cite{Villani} and some variations \cite{baudoin2017bakry,cattiaux2019entropic,Songbo2}. With this approach, in the convex case, sharp rates  were only known for Gaussian targets \cite{M18}.

As a conclusion, inside this active field, by successfully  designing the entropic analogue of the DMS approach (which was an important open problem),  the work~\cite{Lu} of Jianfeng Lu constitutes an important breakthrough.

\section{Proofs}

\subsection{Settings, definitions and notations}

When $\mu \in \mathcal P_2(\R^d)$ has a density with respect to the Lebesgue measure we still write $\mu$ this density, and we write $\mu \propto h$ if it is proportional to  a function $h$. The $\mathcal W_2$ distance over $\mathcal P_2(\R^d)$ is defined by 
\begin{equation}
\label{def:W2}
\mathcal W_2^2(\nu,\mu) = \inf_{\pi\in \mathcal C(\nu,\mu)} \int_{\R^{2d}} |x-y|^2 \pi (\dd x,\dd y)
\end{equation}
where $\mathcal C(\nu,\mu)$ is the set of measures $\pi$ over $\R^d\times\R^d$ having marginal distributions $\nu$ and $\mu$. We denote
\[M_2(\mu) = \int_{\R^d} |x|^2 \mu(\dd x)\]
 the second moment of $\mu\in\mathcal P_2(\R^d)$. For $\mu,\nu \in \mathcal P_2(\R^d)$, the relative entropy and Fisher information of $\nu$ with respect to $\mu$ are respectively defined as
 \[\mathcal H(\nu|\mu) = \int_{\R^d} (h \ln h - h + 1) \mu \,,\qquad \mathcal I(\nu|\mu) = 4 \int_{\R^d} |\na \sqrt{h}|^2 \mu\]
 if $\nu = h\mu$ with $h\in L^1(\mu)$ (these quantities are then well-defined, possibly infinite, since $x\mapsto x\ln x - x +1$ is non-negative and convex over $\R_+$, and $|\na \sqrt{h}|$ is understood as the upper-gradient $\lim_{r\rightarrow 0}\sup_{y:|y-x|\leqslant r}|h(x)-h(y)|/r$), and otherwise as $+\infty$. A measure $\mu \in \mathcal P_2(\R^d)$ is said to satisfy a log-Sobolev inequality with constant $\lambda>0$ with
 \begin{equation}
 \label{eq:LSI}
 \forall \nu \in\mathcal P_2(\R^d),\qquad \mathcal H(\nu|\mu) \leqslant \frac{1}{2\lambda} \mathcal I(\nu|\mu)
 \end{equation}
 and a Talagrand inequality with constant $\lambda$ if 
 \begin{equation}
 \label{eq:Talagrand}
 \forall \nu \in\mathcal P_2(\R^d),\qquad \mathcal W_2^2(\nu,\mu) \leqslant \frac{2}{\lambda} \mathcal H(\nu|\mu)\,.
 \end{equation}
 
When it exists, the linear functional derivative $\frac{\delta \mathcal E}{\delta m}(\rho,\cdot):\R^d\rightarrow \R$ of $\mathcal E:\mathcal P_2(\R^d) \rightarrow (-\infty,\infty]$, defined with \cite[Definition 5.43]{carmona2018probabilistic}, is a function such that
\begin{equation}
\label{eqdef:dE}
\mathcal E(\rho_1) - \mathcal E(\rho_0) = \int_0^1 \int_{\R^d} \frac{\delta \mathcal E}{\delta m}((1-s)\rho_0+s\rho_1,x) (\rho_1-\rho_0)(\dd x) \dd s  
\end{equation}
for all $\rho_0,\rho_1 \in \mathcal P_2(\R^d)$ with $\mathcal E(\rho_0)+\mathcal{E}(\rho_1)<\infty$ (this uniquely defines $\frac{\delta \mathcal E}{\delta m}(\rho,\cdot)$ up to an additive constant). When $x\mapsto \frac{\delta \mathcal E}{\delta m}(\rho,x)$ is $\mathcal C^1$ for any $\rho \in \mathcal P_2(\R^d)$, we write
\begin{equation}
\label{def:DE}
D\mathcal E(\rho,x) = \na_x \frac{\delta\mathcal E}{\delta m}(\rho,x)\,,
\end{equation}
known as the intrinsic derivative of $\mathcal E$.  For a given $x\in\R^d$, $(\rho,y)\mapsto D^2\mathcal E(\rho,x,y)$ is the intrinsic derivative of $\rho \mapsto D\mathcal E(\rho,x)$ (provided it exists), and similarly for higher order derivatives.

We can now state the regularity conditions  used in the proof of Theorem~\ref{thm:main}.

\begin{assumption}\label{assum:regularite_bis}
The energy $\mathcal E $ is  lower-bounded. The intrinsic derivative  $D\mathcal E$ exists and is $\mathcal C^1$ and Lipschitz continuous in its second variable. For all $n\geqslant 2$, $D^n\mathcal E$ exists, is continuous and bounded. For any $R>0$  there exists a constant $K(R)>0$ such that for all $\rho\in\mathcal P_2(\R^d)$ with $M_2(\rho) \leqslant R$,
\begin{equation}
\label{eq:M2'}
\forall \rho'\in\mathcal P_2(\R^d),\qquad \left|\int_{\R^{2d}} \co \frac{\delta^2\mathcal E}{\delta m^2}(\rho',x,x)-\frac{\delta^2\mathcal E}{\delta m^2}(\rho',x,y) \cf \rho(\dd x)\rho(\dd y)\right| \leqslant K(R)\,. 
\end{equation}
\end{assumption}

The Vlasov-Fokker-Planck equation~\eqref{eq:VFP} is well-posed under Assumption~\ref{assum:regularite_bis} for any initial condition $\nu_0 \in \mathcal P_2(\R^d)$, and solutions instantaneously become smooth with finite free energy, see further details and references in the proof of Theorem~\ref{thm:main}.

\begin{example}\label{exemple}
If $\mathcal E(\rho) = \int_{\R^d} V(x) \rho(\dd x) + \int_{\R^{kd}} W(x_1,\dots,x_k) \rho^{\otimes k}(\dd x_1,\dots\dd x_k)$ for some $k\geqslant 2$ then Assumption~\ref{assum:regularite_bis} holds if $V,W$ are lower bounded, $\mathcal C^\infty$ with all derivatives of order $2$ or larger bounded. Moreover, if $V$ and $W$ are convex, then $\mathcal E$ is displacement-convex (see \cite[Section 9.2]{ambrosio2005gradient}).
\end{example}

In fact, since no quantitative bound on the derivatives in Assumption~\ref{assum:regularite_bis}  appears in Theorem~\ref{thm:main}, this condition can be weakened by an approximation argument. Recall that a sequence of functions $(f_n)_{n\in\N}$ from a topological space $E$ to $(-\infty,\infty]$ is said to $\Gamma$-converge to another function $f$ if:
\begin{itemize}
\item For any sequence $(x_n)_{n\in \N}$ in $E$ converging to $x\in E$, $\liminf f_n(x_n) \geqslant f(x)$.
\item For any $x\in E$, there exists a sequence $(x_n)_{n\in \N}$ in $E$ converging to $x$ with $\lim f_n(x_n) =f(x)$.
\end{itemize}
See \cite{ambrosio2005gradient} for further details, properties and applications of this notion.

\begin{assumption}\label{assum:regularite}
There exists a sequence of energies $(\mathcal E_n)_{n\in\N}$ such that:
\begin{itemize}
\item For each $n\in\N$,  $\mathcal E_n$ satisfies Assumptions~\ref{assum:Convexetal} and~\eqref{assum:regularite_bis}.
\item  As $n\rightarrow \infty$,  $\mathcal E_n$ $\Gamma$-converges to $\mathcal E$.
\item For any $\nu_0 \in \mathcal P_2(\R^{2d})$ there exists a sequence $(\nu_0^n)_{n\in\N}$ over $\mathcal P_2(\R^{2d})$ and a trajectory $t\mapsto \nu_t$ in $\mathcal C(\R_+,\mathcal P_2(\R^{2d}))$   such that, as $n\rightarrow \infty$, $\mathcal E_n(\nu_0^n) + \mathcal H(\nu_0^n) \rightarrow \mathcal E(\nu_0) + \mathcal H(\nu_0)$ and  $\nu_t^n$ (the solution of~\eqref{eq:VFP} associated to $\mathcal E_n$) converge for all $t\geqslant 0$ in $\mathcal W_2$ to $\nu_t$. We say that  $(\nu_t)_{t\geqslant 0}$ is a solution of~\eqref{eq:VFP}  (associated to $\mathcal E$) with initial condition $\nu_0$.
\end{itemize}
\end{assumption}

\subsection{The particle system}

For a particle configuration $\bx=(x_1,\dots,x_N)\in(\R^d)^N$ (we will systematically use bold letters for vectors involving $N$ particles), the corresponding empirical distribution and mean-field potential are respectively defined as
\begin{equation}
\pi_{\bx} = \frac1N \sum_{i=1}^N \delta_{x_i},\qquad U_N(\bx) = N \mathcal E\po \pi_{\bx}\pf \,.
\end{equation}
The Langevin process associated to the mean-field potential $U_N$ is the particle system $(\bX_t,\bV_t)_{t\geqslant 0}$ on $(\R^d\times\R^d)^{N}$ solving
\begin{equation}
\label{eq:Langevin}
\left\{
\begin{array}{rcl}
\dd \bX_t &=  & \bV_t \dd t \\
\dd \bV_t & =& -\na U_N(\bX_t) \dd t - \gamma \bV_t\dd t + \sqrt{2\gamma} \dd \mathbf{B}_t\,.
\end{array}
\right.
\end{equation}
We write $\nu_t^N $ the law of $(\bX_t,\bV_t)$, which solves the kinetic Fokker-Planck equation
\begin{equation}
\label{eq:kinFP}
\partial_t \nu_t^N + \bv\cdot \na_{\bx} \nu_t^N - \na U_N(\bx)\cdot \na_{\bv} \nu_t^N = \gamma \co \na_{\bv} \cdot\po   \bv  \nu_t^N \pf +\Delta_{\bv} \nu_t^N\cf\,,
\end{equation}
and $\rho_t^N = \int_{\R^{dN}} \nu_t^N(\cdot,\dd \bv)$ the marginal law of $\bX_t$. 

Under Assumption~\ref{assum:regularite_bis}, $U_N\in\mathcal C^2(\R^{dN},\R)$ with
\[\na_{x_i} U_N(\bx) = D\mathcal E(\pi_{\bx},x_i)\,,\qquad \na^2_{x_i,x_j} U_N(\bx) = \frac1N D^2 \mathcal E(\pi_{\bx},x_i,x_j) + \na D\mathcal E(\pi_{\bx},x_i) \1_{i=j} \]
(which can be deduced from~\eqref{eqdef:dE}) and thus there exists $L>0$ such that for all $N\geqslant 1$,
\begin{equation}
\label{eq:na2UNbounded}
\|\na^2 U_N\|_\infty \leqslant L
\end{equation}
with $\|\na^2 U_N\|_\infty $ the supremum of the operator norm of $\na^2 U(\bx)$ (associated to the Euclidean norm over $\R^{dN}$). In particular this ensures well-posedness of~\eqref{eq:Langevin} and~\eqref{eq:kinFP}.

\begin{lemma}\label{lem:UNconvex}
Under Assumptions~\ref{assum:Convexetal} and~\ref{assum:regularite_bis}, for any $N\geqslant 1$, $U_N$ is convex.
\end{lemma}
\begin{proof}
Let $\bx,\mathbf{v}\in \R^{dN}$ be such that $x_i \neq x_j$  and $|v_i|\leqslant 1$ for all $i\neq j$. Then, for $h>0$ smaller than a quarter of  the smallest distance between the $x_i$'s, the $\mathcal W_2$-optimal transport between $\pi_{\bx- h\mathbf{v}}$ and $\pi_{\bx + h\mathbf{v}}$ is given by the map $T(x_i-hv_i) = x_i + h v_i$, and thus the curve $s \mapsto \pi_{\bx + s \mathbf{v}}$ for $s\in[-h,h]$ is a $\mathcal W_2$-geodesic. Using that $s \mapsto \mathcal E(\pi_{\bx + s \mathbf{v}})$ is convex shows that $ \mathbf{v}\cdot \na^2 U_N(\bx) \mathbf{v}\geqslant 0$. This shows that $\na^2 U_N(\bx) \geqslant 0$ for all $\bx$ with distinct coordinates, and thus for all $\bx \in \R^{dN}$ by density and continuity. 
\end{proof}

It will be convenient to prove the result, as a first step,  in the strongly convex case, adding the following condition:

\begin{assumption}
\label{assum:supplementaire}
The energy is of the form $\mathcal E (\rho) = \mathcal E_0(\rho) + r M_2(\rho)$ for some $r>0$ with an energy $\mathcal E_0$  satisfying  Assumptions~\ref{assum:Convexetal} and~\ref{assum:regularite_bis}. 
\end{assumption}

In the general case, we will add $rM_2(\rho)$ to $\mathcal E(\rho)$ and send $r$ to zero at the end.

Under Assumption~\ref{assum:supplementaire},
\begin{equation}\label{loc:UNquadra}
U_N(\bx) \geqslant r |\bx |^2 + N \inf \mathcal E_0\,.
\end{equation}
Together with the boundedness conditions on derivatives of order higher than $2$ in Assumption~\ref{assum:regularite_bis}, this shows that these two conditions together imply \cite[Assumption 2.2]{Lu} (with $n=1$). This condition is used in \cite{Lu} to ensure  well-posedness and justify computations, relying on hypoelliptic estimates from~\cite{herau2004isotropic}. Moreover,~\eqref{loc:UNquadra} implies that the Gibbs measure $\rho_*^N$ with probability density
\begin{equation}
\label{eq:GIbbs} 
\rho_*^N (\bx) \propto \exp \po - U_N\pf 
\end{equation}
is well-defined and in $\mathcal P_2(\R^{dN})$ for all $N\geqslant 1$. 
The invariant measure of~\eqref{eq:Langevin} is
\[\nu_*^N := \rho_*^N \otimes \kappa^{\otimes N}\,.\]
 Moreover, thanks to Lemma~\ref{lem:UNconvex},
\[\forall \bx,\bv\in\R^{dN},\qquad \bv\cdot \na^2 U_N(\bx)\bv \geqslant  2 r |\bv|^2\,.\]
As a consequence, by the Bakry-Emery criterion~\cite{BakryGentilLedoux}, $\rho_*^N$ (resp. $\nu_*^N$) satisfies a log-Sobolev inequality with constant $2r$ (resp. $\min(2r,1)$, by tensorization), uniformly over $N$. We recover the situation considered in~\cite{MONMARCHE20171721}. More generally, the fact that the log-Sobolev constant is uniform in $N$ has many nice consequences and have been extensively studied. Let us summarize the known results that will be useful for the rest of the study.

\begin{proposition}
\label{prop:old_results}
Under Assumption~\ref{assum:supplementaire}, there exists $C_*>0$, $a\in(0,1]$ such that:
\begin{enumerate}
\item Any solutions of~\eqref{eq:VFP} and~\eqref{eq:kinFP} with respective initial data $\nu_0 \in \mathcal P_2(\R^{2d})$ and $\nu_0^{N} \in \mathcal P_2(\R^{2dN})$ satisfy
\begin{equation}
\label{eq:PoCunif}
\forall t\geqslant 0,\qquad  \mathcal W_2^2(\nu_t^{\otimes N},\nu_t^N) \leqslant C_* e^{C_* t} \po 1+M_2(\nu_0)+\mathcal W_2^2(\nu_0^{\otimes N},\nu_0^N)\pf \,.
\end{equation}
\item For all $N\geqslant 1$, for the stationary solutions of \eqref{eq:VFP}  and~\eqref{eq:kinFP},
\begin{equation}
\label{eq:PoC*}
\mathcal W_2^2(\nu_*^{\otimes N},\nu_*^N) + \mathcal H(\nu_*^{\otimes N}|\nu_*^N) + \mathcal I(\nu_*^{\otimes N}|\nu_*^N) \leqslant C_* N^{1-a}  \,. 
\end{equation}
\item For any $\nu_0 \in \mathcal P_2(\R^{2d})$ with $\FF_k(\nu_0) < \infty$,
\begin{equation}
\label{eq:GrandesDev_nu0}
\frac1N \mathcal H(\nu_0^{\otimes N} | \nu_*^N) \underset{N\rightarrow \infty} \longrightarrow \FF_k(\nu_0)\,.
\end{equation}
\item Let  $\nu_0 \in \mathcal P_2(\R^{2d})$ with $\FF_k(\nu_0)<\infty$ and, for all $N\geqslant 1$, let $\nu_0^N \in \mathcal P(\R^{2dN})$ be an exchangeable probability measure. Set
\begin{equation}
\label{eq:eN}
e_N(t)= \po N \FF(\rho_t) -  \mathcal H(\rho_t^N|\rho_*^N)  \pf_+ \,,\qquad f_N(t)= \po N \FF_k(\nu_t) -  \mathcal H(\nu_t^N|\nu_*^N)  \pf_+
\end{equation}
with $\nu_t$ and $\nu_t^N$ the position marginals of the solutions of~\eqref{eq:VFP} and~\eqref{eq:kinFP} with respective initial data $\nu_0$ and $\nu_0^{ N}$, and $\rho_t$ and $\rho_t^N$ their position marginals. If $\sup_{N\geqslant 1}\mathcal W_2^2 (\nu_0^{\otimes N},\nu_0^N) <\infty$ then, for any $t>0$,
\begin{equation}
\label{eq:eN(t)}
\liminf_{N\rightarrow \infty} \frac1N  \co e_N(t) + f_N(t) +  \int_0^t \po e_N(s) + f_n(s)\pf \dd s\cf   = 0\,.
\end{equation}
\end{enumerate}
\end{proposition}

\begin{proof}
\emph{Item 1.} Exploiting~\eqref{eq:na2UNbounded} and Assumption~\ref{assum:regularite_bis}, the bound~\eqref{eq:PoCunif} follows from a classical synchronous coupling argument, see  \cite[Proposition 5.2]{Songbo2}.

\medskip

\emph{Item 2.} Under Assumption~\ref{assum:supplementaire}, $\nu_*^N$ satisfies log-Sobolev and Talagrand inequalities uniform in $N$, which means that we only have to bound $\mathcal I(\nu_*^{\otimes N}|\nu_*^N)$ to get~\eqref{eq:PoC*}. Since $\rho_*$ is a critical point of $\FF$, it satisfies $\na \ln \rho_* = - D\mathcal E(\rho_*,\cdot)$, hence
\begin{align*}
\mathcal I(\nu_*^{\otimes N}|\nu_*^{N}) &= \int_{\R^{dN}} \left|\na \ln \rho_*^{\otimes N} + \na U_N\right|^2 \rho^{\otimes N}\\
&=  \sum_{i=1}^N \int_{\R^{dN}}|D\mathcal E(\rho_*,x_i) - D\mathcal E(\pi_{\bx},x_i)|^2 \rho_*^{\otimes N}(\bx) \dd\bx  \\
&\leqslant N \|D^2 \mathcal E\|_\infty^2 \int_{\R^{dN}}\mathcal W_2^2( \rho_*,\pi_{\bx}) \rho_*^{\otimes N}(\bx) \dd\bx\,.
\end{align*}
As a conclusion,~\eqref{eq:PoC*} follows from \cite[Theorem 1]{FournierGuillin}.

\medskip

\emph{Item 3.} Thanks to~\eqref{eq:PoC*}, we already know the result for $\nu_0^{\otimes N} = \nu_*^{\otimes N}$,  since
\[0 \leqslant \frac1N \mathcal H(\nu_0^{\otimes N} |\nu_*^N) =  \int_{\R^{d}} \rho_* \ln \rho_* +   \int_{\R^{dN}} \mathcal E(\pi_{\bx}) \rho_*^{\otimes N}  + \frac1N \ln \int_{\R^{dN}} e^{-U_N} \leqslant C_* N^{-a}  \underset{N\rightarrow\infty}\longrightarrow  0\,. \]
Since $\FF(\rho_*)=0$, this amounts to 
\begin{equation}
\label{eqloc:0leq}
0 \leqslant   \int_{\R^{dN}}   \mathcal E(\pi_{\bx})  \rho_*^{\otimes N} - \mathcal E(\rho_*)  + \frac1N \ln \int_{\R^{dN}} e^{-U_N} \leqslant C_* N^{-a}  \underset{N\rightarrow\infty}\longrightarrow  0\,.
\end{equation}
From~\eqref{eq:M2'}, following \cite[Theorem 4.2.9 (i)]{tse2019quantitative}, for any $\rho \in \mathcal P_2(\R^{d})$,
\begin{equation}
\label{loc:Tse}
\left| \int_{\R^{dN}} \ \mathcal E(\pi_{\bx})    \rho^{\otimes N} - \mathcal E(\rho)\right| \leqslant \frac{K(M_2(\rho))}{2N}\,.
\end{equation}
Using this with $\rho = \rho_*$ and~\eqref{eqloc:0leq} this means that
\begin{equation}
\label{loc:ZN}
 \frac1N \ln \int_{\R^{dN}} e^{-U_N}  \underset{N\rightarrow\infty}\longrightarrow  0\,.
\end{equation}
 Now, for any $\nu_0 \in \mathcal P_2(\R^{2d})$,
 \begin{align*}
 \frac1N \mathcal H(\nu_0^{\otimes N} | \nu_*^N)  &=  \mathcal H(\nu_0) +   \int_{\R^{dN}} \mathcal E(\pi_{\bx}) \rho_0^{\otimes N} + \int_{\R^{2d}}\frac{|v|^2}{2}\nu_0   + \frac1N \ln \int_{\R^{dN}} e^{-U_N} + \frac{d}{2}\ln(2\pi) \\
 & \underset{N\rightarrow \infty}\longrightarrow \mathcal F_k(\nu_0)\,,
\end{align*}
 where we used~\eqref{loc:Tse} and~\eqref{loc:ZN}.

\medskip

\emph{Item 4.} Since $e_N(t)/N \leqslant \FF(\rho_t) \leqslant \FF_k(\nu_t) \leqslant \FF_k(\nu_0) $, by dominated convergence, it is sufficient to prove that $\liminf e_N(t)/N =0$ for any $t\geqslant 0$ (and similarly for $f_N$). Thanks to~\eqref{eq:PoCunif},
\begin{equation}
\label{loc:pocProof}
\mathcal W_2^2 (\rho_t^{\otimes N},\rho_t^N) \leqslant \mathcal W_2^2 (\nu_t^{\otimes N},\nu_t^N) = \underset{N\rightarrow \infty}O(1)\,.
\end{equation}
By interchangeability, this implies that the one-particle marginal $\rho_t^{N,1}$ of $\rho_t^N$ converges to $\rho_t$ in $\mathcal W_2$ as $N\rightarrow \infty$.

We decompose
\begin{align*}
 \mathcal H(\rho_t^N|\rho_*^N) &=  \mathcal H(\rho_t^N|\kappa^{\otimes N}) - \frac{Nd}{2}\ln(2\pi) - \frac12 M_2(\rho_t^N) + \int_{\R^{dN}} U_N \rho_t^N  + \ln\int_{\R^{dN}} e^{-U_n}\,.
\end{align*}
By the additivity property of the relative entropy with respect to a product probability measure (see \cite[Lemma 18]{GuillinWuZhang}), the interchangeability of $\rho_t^N$ and~\eqref{loc:ZN},
\begin{align*}
\frac1N  \mathcal H(\rho_t^N|\rho_*^N)  &\geqslant \mathcal H(\rho_t^{N,1}|\kappa) - \frac{d}{2}\ln(2\pi) - \frac12 M_2(\rho_t^{N,1}) + \int_{\R^{dN}} \mathcal E(\pi_{\bx}) \rho_t^N + \underset{N\rightarrow \infty}o(1)\\
&= \mathcal H(\rho_t^{N,1}) + \int_{\R^{dN}} \mathcal E(\pi_{\bx}) \rho_t^N + \underset{N\rightarrow \infty}o(1)\,.
\end{align*}
As $N \rightarrow \infty$, since $\rho_t^{N,1} \rightarrow \rho_t$ and the entropy is lower semi-continuous, $\liminf \mathcal H(\rho_t^{N,1}) \geqslant \mathcal H(\rho_t)$. To conclude, it only remains to show that
\begin{equation}
\label{loc:epixepiy}
\left|\int_{\R^{dN}} \mathcal E(\pi_{\bx}) \rho_t^N -  \mathcal E(\rho_t)\right| \underset{N\rightarrow \infty}\longrightarrow 0\,.
\end{equation}
Let $\bX=(X_1,\dots,X_N)$ and $\bY=(Y_1,\dots,Y_N)$ be an optimal coupling of $\rho_t^N$ and $\rho_t^{\otimes N}$, so that
\[\mathbb E\po |\bX - \bY|^2\pf = \mathcal W_2^2(\rho_t^N,\rho_t^{\otimes N})  = \underset{N\rightarrow \infty}O(1) \,.\]
Then, using~\eqref{loc:Tse} with $\rho=\rho_t$, thanks to~\eqref{eq:na2UNbounded},
\begin{align*}
\left|\int_{\R^{dN}} \mathcal E(\pi_{\bx}) \rho_t^N -  \mathcal E(\rho_t)\right| & \leqslant \frac1N\left|\mathbb E \po U_N(\bX) - U_N(\bY)\pf \right| + \underset{N\rightarrow \infty}o(1) \\
& \leqslant \frac1N \mathbb E \po |\na U_N(\bY)||\bX-\bY|\pf   + \underset{N\rightarrow \infty}o(1) \\
& \leqslant \frac1N \sqrt{ 2\mathbb E \po |\na U_N(0)|^2 + L |\bY|^2 \pf \mathbb E \po|\bX-\bY|^2\pf }   + \underset{N\rightarrow \infty}o(1) \\
& \leqslant \frac{\sqrt{2}}{\sqrt{N}} \po   |D\mathcal E(\delta_0,0)| + \sqrt{L M_2(\rho_t)} \pf \sqrt{\mathbb E \po|\bX-\bY|^2\pf }   + \underset{N\rightarrow \infty}o(1) \\
&= \underset{N\rightarrow \infty}o(1) \,.
\end{align*}
This concludes the proof for $e_N$. The proof for $f_N$ is the same mutatis mutandis (i.e. replacing $\mathcal E$ by $\nu \mapsto \mathcal E(\rho)+\frac12\int_{\R^{2d}} |v|^2 \nu(\dd x\dd v)$ and exploiting the propagation of chaos~\eqref{loc:pocProof} for $\nu_t^N$ instead of simply $\rho_t^N$).
\end{proof}

The key point in order to apply the approach of \cite{Lu} with the optimal PL constant $\lambda_*$ and without having to establish or assume a uniform-in-$N$ LSI with constant $\lambda_*$ is the following approximate Talagrand inequality for the Gibbs measure $\rho_*^N$.

\begin{proposition}
\label{prop:approxTalagrand}
Under Assumption~\ref{assum:supplementaire}, for any $\nu_0 \in \mathcal P_2(\R^{2d})$ with $\FF_k(\nu_0)<\infty$ and any $(\nu_0^N)_{N\geqslant 1}$ with $\nu_0^N \in \mathcal P_2(\R^{2dN})$ and such that $\sup_{N\geqslant 1}\mathcal W_2(\nu_0^{\otimes N},\nu_0^N) <\infty$, there exists $C_0>0$ such that for all $N\geqslant 1$ and all $t\geqslant 0$, using the notations of Proposition~\ref{prop:old_results} and recalling that $\rho_t,\rho_t^N$ are the position marginals of $\nu_t,\nu_t^N$,
\begin{equation}
\label{eq:approxTalagrand}
\mathcal W_2^2(\rho_t^N,\rho_*^N) \leqslant \frac2{\lambda_*} \mathcal H(\rho_t^{N}|\rho_*^N) + R_N(t),
\end{equation}
with $R_N(t) = C_0 e^{C_* t} N^{1-a/2} + 2e_N(t)/\lambda_*$.
\end{proposition}
\begin{proof}
By triangular inequality, applying then~\eqref{eq:PoCunif}, \eqref{eq:PoC*}, the scaling property of $\mathcal W_2$ for tensor products  and then \eqref{eq:lambdaPL}, for some constant $C_0'>0$ independent from $N$,
\begin{align*}
\mathcal W_2(\rho_t^N,\rho_*^N)& \leqslant \mathcal W_2(\rho_t^N,\rho_t^{\otimes N}) + \mathcal W_2(\rho_t^{\otimes N},\rho_*^{\otimes N}) + \mathcal W_2(\rho_*^{\otimes N},\rho_*^N) \\
&\leqslant \sqrt{C_0' e^{C_* t} (1+N^{1-a})} + \sqrt{N} \mathcal W_2(\rho_t,\rho_*) \\
&\leqslant \sqrt{C_0' e^{C_* t}(1+N^{1-a})} + \sqrt{2 N \FF(\rho_t)/\lambda_*} \,.
\end{align*} 
Taking the square and using that $\FF(\rho_t) \leqslant \FF_k(\nu_t)\leqslant \FF_k(\nu_0) $  yields
\[
\mathcal W_2^2(\rho_t^N,\rho_*^N) ,\leqslant C_0' e^{C_* t}(1+N^{1-a})+ 2\sqrt{2C_0' e^{C_* t}(1+N^{1-a})   N \FF_k(\nu_0) /\lambda_*} + \frac{2}{\lambda_*} N \FF(\rho_t),\]
which gives~\eqref{eq:approxTalagrand} thanks to the definition~\eqref{eq:eN} of $e_N$.
\end{proof}

\begin{remark}\label{rem:lambda(t)}
It is clear from the proof that, for a given $t\geqslant0$, the inequality~\eqref{eq:approxTalagrand} still holds if we replace $\lambda_*$ by $\lambda(t):= 2\FF(\rho_t)/\mathcal W_2^2(\rho_t,\rho_*)$.
\end{remark}

\subsection{Proof of Theorem~\ref{thm:main}}
The argument is simply to follow the proof of \cite[Theorem 1]{Lu} and to replace the Talagrand inequality by~\eqref{eq:approxTalagrand} whenever necessary (i.e. in \cite[Lemma 3.2]{Lu}). We will not repeat line by line the whole argument but simply highlight the change, and refer to \cite{Lu} for details (such as  the well-posedness of the quantities involved; in particular, we use the notion of regular solutions of~\eqref{eq:kinFP} from  \cite[Definition 2.5]{Lu}, which is used to justify  time derivations).

We decompose the proof in 6 steps. From steps 1 to 4, the strong Assumption~\ref{assum:supplementaire} is enforced. In Step 5, the strong convexity condition is removed by letting $r\rightarrow 0$. In Step 6, the regularity Assumption~\ref{assum:regularite_bis} is removed by letting $n\rightarrow \infty$ in Assumption~\ref{assum:regularite}.

\paragraph{Step 1: results from~\cite{Lu}.} Let us first introduce some key notions from this work. For a density $\nu \in \mathcal P_2(\R^{2dN})$ with finite entropy, denote by $g$ its density with respect to $\nu_*^N$, by $q(\bx)=\int_{\R^{dN}} g(\bx,\bv)\kappa^{\otimes N}(\dd \bv)$ the density of its position marginal $\rho$ with respect to $\rho_*^N$, and introduce the Wasserstein current corrector
\[\COT(\nu) = \int_{\R^{2dN}} \bv \cdot \po \bx - T_{q}(\bx)\pf  \nu(\dd \bx,\dd\bv)\dd \bx\dd \bv \,, \]
where $T_q$ is the optimal transport map from $\rho = q \rho_*^N$ to $\rho_*^N$. Define further  the average current energy
\[J(\nu)  = \int_{\R^d} \frac{1}{q(\bx)} \left|\int_{\R^d} \bv g(\bx,\bv)\kappa(\bv) \right|^2 \rho_* ^N(\bx)\dd x, \]
and the fiber entropy and velocity Fisher information
\[\mathcal H_v(\nu) = \int_{\R^{dN}} \mathcal H\po \left. \frac{\nu(\bx,\cdot)}{\rho(\bx)}\right |\kappa^{\otimes N} \pf \rho(\bx)\dd \bx  \,,\qquad  \mathcal I_v(\nu) = \int_{\R^{2dN}} \left|\na_{\bv} \ln \nu(\bx,\bv) + \bv\right|^2 \nu(\bx,\bv)\dd \bx\dd \bv\,.\]
 Finally, for $\varepsilon>0$, define the modified entropy as
\[\Heps(\nu) = \mathcal H(\nu|\nu_*^N) + \varepsilon \COT(\nu)\,.\]
Then, it is established in~\cite[Lemmas 3.1 and 3.2]{Lu} that
\begin{equation}
\label{eq:Lu1}
J(\nu) \leqslant \min \po 2\mathcal H_v(\nu),\mathcal I_v(\nu)\pf \qquad  |\COT(\nu)| \leqslant \sqrt{J(\nu)} \mathcal W_2(\rho,\rho_*^N)  
\end{equation}
and in \cite[Equation (6.3)]{Lu} that, along regular solutions of~\eqref{eq:kinFP},
\begin{equation}
\label{loc:Lu2}
\partial_t \Heps (\nu_t^N) \leqslant - \co (\gamma - 3 \varepsilon) \mathcal I_v(\nu_t^N) + \varepsilon \mathcal H(\rho_t^N|\rho_*^N) + \varepsilon \gamma \COT(\nu_t^N)\cf \,.
\end{equation}

\paragraph{Step 2: using the approximate Talagrand inequality.} For now, consider initial conditions as follows: for $\nu_0\in \mathcal P_2(\R^{2d})$ satisfying $\mathcal I_k(\nu_0) < \infty $, we let $\nu_0^N = (1-q_N) \nu_0^{\otimes N} + q_N \nu_*^N$ with $q_N = N^{-2} $. This choice ensures that
\[  \mathcal W_2^2 (\nu_0^N,\nu_0^{\otimes N}) \leqslant q_N\mathcal W_2^2 (\nu_*^{N},\nu_0^{\otimes N}) \leqslant 2 q_N  \po \mathcal W_2^2 (\nu_*^{N},\nu_*^{\otimes N}) + N \mathcal W_2^2 (\nu_*,\nu_0)\pf \underset{N\rightarrow \infty}\longrightarrow 0, \]
thanks to~\eqref{eq:PoC*}.  Since $\nu_*^N$ is invariant for the Markov process~\eqref{eq:Langevin}, $\nu_t^N \geqslant q_N \nu_*^N$ for all $N\in \N$ which, together with the hypoellipticity estimates from \cite{herau2004isotropic}, implies that $(\nu_t^N)_{t\geqslant t_0}$ is a regular solution of~\eqref{eq:kinFP} for any $t_0>0$, as shown in \cite[Lemmas 7.1 and 7.2]{Lu}.

Exploiting~\eqref{eq:Lu1} and Young's inequality yields
\[| \Heps(\nu_t^N) - \mathcal H(\nu_t^N |\nu_*^N) | \leqslant \varepsilon a \mathcal H_v(\nu_t^N) + \frac{\varepsilon}{2a} \mathcal W_2^2(\rho_t^N,\rho_*^N) \]
for any $a>0$.  This is the first  part where the Talagrand inequality for $\rho_*^N$ is used in \cite{Lu}.  Instead, we use its approximation~\eqref{eq:approxTalagrand}, obtaining
\begin{align*}
| \Heps(\nu_t^N) - \mathcal H(\nu_t^N|\nu_*^N) | & \leqslant \varepsilon a \mathcal H_v(\nu_t^N) + \frac{\varepsilon }{a \lambda_*} \mathcal H(\rho_t^N|\rho_*^N) +  \frac{\varepsilon}{2a} R_N(t)\,. 
\end{align*}
Using that $\mathcal H_v(\nu_t^N)  + \mathcal H(\rho_t^N|\rho_*^N) = \mathcal H(\nu_t^N|\nu_*^N) $ by classical decomposition of the relative entropy, we choose $a=\lambda_*^{-1/2}$ to get 
\begin{equation}
\label{loc:equivalenceH}
| \Heps(\nu_t^N) - \mathcal H(\nu_t^N|\nu_*^N) |  \leqslant \frac{\varepsilon}{\sqrt{\lambda_*}} \mathcal H(\nu_t^N|\nu_*^N) +  \frac{\varepsilon}{2a} R_N(t)\,.
\end{equation}

We now turn to~\eqref{loc:Lu2}. Using~\eqref{eq:Lu1} to bound the last term gives
\begin{align*}
\partial_t \Heps (\nu_t^N) &\leqslant - \co (\gamma - 3 \varepsilon) \mathcal I_v(\nu_t^N) + \varepsilon \mathcal H(\rho_t^N|\rho_*^N) - \varepsilon \gamma \sqrt{\mathcal I_v(\nu_t^N)} \mathcal W_2(\rho_t^N,\rho_*^N)\cf\\
& \leqslant - \co (\gamma - 3 \varepsilon - \varepsilon \gamma a) \mathcal I_v(\nu_t^N) + \varepsilon \mathcal H(\rho_t^N|\rho_*^N) - \frac{\varepsilon \gamma}{4a}  \mathcal W_2^2(\rho_t^N,\rho_*^N)\cf 
\end{align*}
for any $a>0$ by Young's inequality. This is the second part where the Talagrand inequality for $\rho_*^N$ is used in \cite{Lu}, and again we use~\eqref{eq:approxTalagrand} instead to obtain, for all $t>0$,
\begin{equation}
\label{loc:dfgdfg}
\partial_t \Heps (\nu_t^N) \leqslant - \co (\gamma - 3 \varepsilon - \varepsilon \gamma a) \mathcal I_v(\nu_t^N) + \varepsilon \po 1 - \frac{\gamma}{2a \lambda_*}\pf \mathcal H(\rho_t^N|\rho_*^N) - \frac{\varepsilon \gamma}{4a}  R_N(t)\cf \,.
\end{equation}

\paragraph{Step 3: closing the modified entropy dissipation inequality.} From now on we can simply follow \cite{Lu}, without worrying about the additional error terms which will vanish at the end. Focusing on the two first terms of~\eqref{loc:dfgdfg} and treating them as in \cite{Lu}, we are led to define $\Gamma = \gamma/\sqrt{\lambda_*} $ and $\theta = \varepsilon/\sqrt{\lambda_*} $ and chose $a=1/(4\varepsilon)$  to get
\[\partial_t \Heps (\nu_t^N) \leqslant - \sqrt{\lambda_*} \co \po \frac34 \Gamma -  3\theta \pf  \mathcal I_v(\nu_t^N) + \theta \po 1 - 2\theta \Gamma\pf \mathcal H(\rho_t^N|\rho_*^N) - \frac{\varepsilon \gamma}{4a}  R_N(t)\cf \,.\]
Taking $\theta$ as  in the statement of Theorem~\ref{thm:main}, i.e.
\begin{equation}
\label{eq:interv_theta}
0 < \theta \leqslant \min \po \frac{\Gamma}{12}, \frac{1}{4\Gamma} \pf ,
\end{equation}
and using that $\mathcal I(\nu_t^N) \geqslant 2 \mathcal H_v(\nu_t^N)$ due to the log-Sobolev inequality for $\kappa^{\otimes N}$, we end up with 
\begin{align*}
\partial_t \Heps (\nu_t^N) & \leqslant - \sqrt{\lambda_*} \co    \Gamma  \mathcal H_v(\nu_t^N) + \frac{\theta}{2} \mathcal H(\rho_t^N|\rho_*^N) - \frac{\varepsilon \gamma}{4a}  R_N(t)\cf \\
& \leqslant - \sqrt{\lambda_*} \frac{\theta}{2}  \mathcal H(\nu_t^N|\nu_*^N) + \tilde C_1  R_N(t)
\end{align*}
where we used again the decomposition of the entropy and set $\tilde C_1 = \frac{\sqrt{\lambda_*} \varepsilon \gamma}{4a} $. Thanks to~\eqref{loc:equivalenceH}, we end up with
\begin{equation}\label{eqloc:dtH}
\partial_t \Heps (\nu_t^N) \leqslant - \sqrt{\lambda_*}\eta   \Heps (\nu_t^N) + \tilde C_2  R_N(t)\,,\quad \eta = \frac{\theta}{2(1+\theta)}\,,\quad \tilde C_2 = \tilde C_1 + \eta \sqrt{\lambda_*} \frac{\varepsilon}{2a}  \,.
\end{equation}

\paragraph{Step 4: conclusion (with the additional condition).}
Integrating~\eqref{eqloc:dtH} over times $s\in[t_0,t]$ for some $t_0>0$  and using twice the (approximate) equivalence between $\Heps (\nu_t^N)$ and $\mathcal H(\nu_t^N|\nu_*^N)$ provided by \eqref{loc:equivalenceH} (since $\theta<1$),
\begin{align*}
(1-\theta) \mathcal H(\nu_t^N|\nu_*^N) & \leqslant \Heps(\nu_t^N) + \frac{\varepsilon}{2a} R_N(t) \\
& \leqslant e^{-\sqrt{\lambda_*}\eta (t-t_0)} \Heps(\nu_{t_0}^{N}) + \tilde C_2 \int_0^t R_N(s)\dd s  + \frac{\varepsilon}{2a}  R_N(t) \\  
& \leqslant e^{-\sqrt{\lambda_*}\eta (t-t_0)} (1+\theta) \mathcal H(\nu_{t_0}^{N}|\nu_*^N)+(1-\theta) \tilde  R_N(t)\,.
\end{align*}
with $\tilde R_N(t) =  (\tilde C_2 \int_0^t R_N(s)\dd s  + \frac{\varepsilon}{2a}  R_N(t))/(1-\theta)$. Using that $t_0>0$ is arbitrary, that $t\mapsto \mathcal H(\nu_t^N|\nu_*^N)$ is decreasing and that $\mathcal H(\nu_{0}^{N}|\nu_*^N) \leqslant (1-q_N) \mathcal H(\nu_{0}^{\otimes N}|\nu_*^N)$ by Jensen inequality and convexity of $h\mapsto h\ln h$ over $\R_+$, we end up with
\begin{equation}
\label{eqloc:1-thetaH}
 \mathcal H(\nu_t^N|\nu_*^N) \leqslant  e^{-\sqrt{\lambda_*}\eta t} \frac{1+\theta}{1-\theta} \mathcal H(\nu_{0}^{\otimes N}|\nu_*^N)+ \tilde   R_N(t)\,.
\end{equation}
%
Thanks to~\eqref{eq:GrandesDev_nu0} and \eqref{eq:eN(t)}, dividing by $N$ and taking the $\liminf$ gives
\begin{equation}
\label{eqloc:Fkconclusion}
 \FF_k(\nu_t) \leqslant \liminf_{N\rightarrow \infty} \frac{\mathcal H(\nu_t^N|\nu_*^N) + f_N(t)}N   \leqslant e^{-\sqrt{\lambda_*}\eta t} \frac{1+\theta}{1-\theta} \FF_k(\nu_0)\,.
\end{equation}
For now we have established this for $\nu_0\in\mathcal P_2(\R^{2d})$ with $\mathcal I_k(\nu_0) < \infty$, however, thanks to \cite[Propositions 5.4 and 5.5]{Songbo2} and \cite[Equation (79)]{monmarche2025local}, we have that $\mathcal F_k(\nu_t)$ and $\mathcal I_k(\nu_t)$ are finite for all $t>0$ for any $\nu_0\in\mathcal P_2(\R^{2d})$. Applying the previous inequality with an arbitrary initial time $t_0>0$ and using that $t\mapsto \mathcal F_k(\nu_t)$ is non-increasing shows that~\eqref{eqloc:Fkconclusion} holds for any $\nu_0 \in \mathcal P_2(\R^{2d})$.

\paragraph{Step 5: removing the strong convexity condition.} Next, we have to get rid of the additional Assumption~\ref{assum:supplementaire} (with Assumption~\ref{assum:regularite_bis} still enforced). This is done by considering the energy $\mathcal E^r(\rho) = \mathcal E(\rho) +r M_2( \rho)$ for an arbitrary $r>0$. By convexity, the corresponding free energy $\FF^r$ admits a unique global minimizer $\rho_{*}^r$, and then the re-centered free energy $\rho \mapsto \FF^r(\rho) - \FF^r(\rho_{*}^r)$  satisfies Assumption~\ref{assum:supplementaire}. We deduce from the beginning of the proof that 
\begin{equation}\label{loc:Fkr}
 \FF_k^{r}(\nu_t^{r}) - \FF_k^{r}(\nu_*^r) \leqslant e^{-\sqrt{\lambda_*^{r}}\eta t} \frac{1+\theta}{1-\theta} \po \FF_k^{r}(\nu_0) - \FF_k^{r}(\nu_*^r) \pf
\end{equation}
for all $t\geqslant 0$, where the superscripts $r$ mean that these correspond to $\mathcal E^r$. For any $\rho \in \mathcal P_2(\R^d)$, as $r\rightarrow 0$, $\mathcal F^{r}(\rho) \rightarrow \FF(\rho)$, and moreover if $\rho^r \rightarrow \rho$ in $\mathcal P_2(\R^d)$ then $M_2(\rho^r)\rightarrow M_2( \rho)$ and thus, by the lower semi-continuity of $\mathcal F$, $\liminf_{r\rightarrow 0} \mathcal F^r(\rho^r) \geqslant \mathcal F(\rho)$. These two properties mean by definition that $\FF^r $ $\Gamma$-converges to $\FF$ (and thus similarly for $\FF_k^r$ to $\FF_k$). Moreover, using that $\rho_*^r$ and $\rho_*$ are the global minimizers respectively of $\FF^r$ and $\FF$,
\[r m_2(\rho_*^r) = \FF(\rho_*)+r m_2(\rho_*^r)   \leqslant \FF(\rho_*^r) + rm_2(\rho_*^r) = \FF^r(\rho_*^r) \leqslant \FF^r(\rho_*) =  rm_2(\rho_*)\,. \]
As a consequence, $(\rho_*^r)_{r\in(0,1]}$ is tight in $\mathcal P_2(\R^d)$. By $\Gamma$-convergence and uniqueness of the minimizer of $\FF$, we get $\rho_*^r \rightarrow \rho_*$ in $\mathcal P_2(\R^d)$ as $r\rightarrow 0$. Moreover, the previous bounds
\[r m_2(\rho_*^r) \leqslant  \FF^r(\rho_*^r)  \leqslant r m_2(\rho_*) \]
show that
\begin{equation}
\label{loc:FFr*->}
\FF^r(\rho_*^r)   \underset{r\rightarrow 0}\longrightarrow 0.
\end{equation}
As a consequence,
\begin{equation}
\label{loc:FFkr}
  \FF_k^{r}(\nu_0) - \FF_k^{r}(\nu_*^r)  \underset{r\rightarrow 0}{\longrightarrow} \FF_k(\nu_0)\,. 
\end{equation}
Next, by considering for $r\geqslant 0$ the kinetic McKean-Vlasov processes
\begin{equation}
\label{eq:LangevinNLr}
\left\{
\begin{array}{rcl}
\dd  X_t^r &=  &  V_t^r \dd t \\
\dd  V_t^r & =& - 2 r X_t^r - D\mathcal E(\rho_t^r ,  X_t^r) \dd t - \gamma  V_t^r\dd t  + \sqrt{2\gamma} \dd B_t
\end{array}
\right.
\end{equation} 
(with the same Brownian motion $B$ for all $r$) with initial condition $(X_0^r,Y_0^r) = (X_0,Y_0) \sim \nu_0$, we have that $(X_t^r,V_t^r)\sim \nu_t^r$ and $(X_t^0,V_t^0)\sim \nu_t$ for all $t\geqslant 0$, so that
\[\mathcal W_2^2(\nu_t^r,\nu_t) \leqslant \mathbb E \po |X_t^r - X_t|^2 + |V_t^r - V_t|^2\pf\,. \]
By the Grönwall lemma, uniform bounds on $M_2(\rho_s)$ over $s\in[0,t]$ and the regularity bounds from Assumption~\ref{assum:regularite_bis} it is straightforward to get that
\begin{equation}
\label{loc:unifWr}
\sup_{s\in[0,t]} \mathcal W_2(\nu_s^r,\nu_s) \underset{r\rightarrow 0}\longrightarrow 0
\end{equation}
for all $t\geqslant 0$. In particular, by $\Gamma$-convergence and~\eqref{loc:FFr*->},
\begin{equation}
\label{loc:FFkr2}
 \liminf_{r\rightarrow 0} \FF_k^{r}(\nu_t^{r}) - \FF_k^{r}(\nu_*^r) \geqslant \FF_k(\nu_t)\,.
\end{equation}
Finally, thanks to Remark~\ref{rem:lambda(t)}, we see that~\eqref{loc:Fkr} still holds if we replace $e^{-\sqrt{\lambda_*^{r}}\eta t}$ by
\[\exp \po - \eta \int_0^t \sqrt{\lambda^r(s)}\dd s \pf,\qquad \lambda^r(s) = \frac{2(\FF^r(\rho_s^r)- \FF^r(\rho_{*}^r))}{\mathcal W_2^2(\rho_s^r,\rho_*^r)}\,.\]
For any $\delta>0$, there exist $r_0>0$ such that for all $s\in [0,t]$ and $r\in(0,r_0]$,
\[\lambda^r(s) \geqslant \frac{2\FF^r(\rho_s)}{\delta+\mathcal W_2^2(\rho_s,\rho_*)} - \delta =: \lambda_\delta^0(s)\,.\]
Indeed, if it were not the case, we could find for some $\delta>0$ a sequence  $(s_k,r_k)_{k\in\N}$ on $[0,t]\times (0,r_0]$ with $r_k \rightarrow 0$ and (up to extracting a subsequence) $s_k \rightarrow s $ for some $s\in[0,t]$ such that $\lambda^{r_k}(s_k) < \lambda^0_\delta(s_k)  $ for all $k$. Using the uniform convergence~\eqref{loc:unifWr}, that $\lambda_\delta^0$ is continuous and the $\Gamma$-convergence of $\FF^r$, we would obtain a contradiction with $\lambda_\delta^0(s)  \geqslant \liminf_{r\rightarrow 0} \lambda^{r_k}(s_k) \geqslant \lambda_0^0(s)$. Since $\delta>0$ is arbitrary, we get that 
\[\limsup_{r\rightarrow 0}  \exp \po - \eta \int_0^t \sqrt{\lambda^r(s)}\dd s\pf \leqslant  \exp \po - \eta \int_0^t \sqrt{\lambda^0_0(s)}\dd s\pf \,. \]
Combining this with~\eqref{loc:FFkr2} and~\eqref{loc:FFkr}, we get 
\begin{align*}
\FF_k(\nu_t) & \leqslant  \liminf_{r\rightarrow 0} \FF_k^{r}(\nu_t^{r}) - \FF_k^{r}(\nu_*^r) \\
& \leqslant  \liminf_{r\rightarrow 0}  \exp \po - \eta \int_0^t \sqrt{\lambda^r(s)}\dd s\pf  \frac{1+\theta}{1-\theta} \po \FF_k^{r}(\nu_0) - \FF_k^{r}(\nu_*^r) \pf \\
&\leqslant \exp \po - \eta \int_0^t \sqrt{\lambda^0_0(s)}\dd s\pf  \frac{1+\theta}{1-\theta}  \FF_k(\nu_0)\,.
\end{align*}
This concludes the proof of Theorem~\ref{thm:main} (since $\lambda^0_0(s) \geqslant \lambda_*$ for all $s\geqslant 0$).

 \paragraph{Step 6: last approximation.} This is similar to the previous step: thanks to the previous proof, the result holds for $\mathcal E_n$ for all $n\in \N$ and all the conditions in Assumption~\ref{assum:regularite} have been set to get the desired result as  $n\rightarrow \infty$ by $\Gamma$-convergence.

\subsection*{Acknowledgments}

The research of PM is supported by the project CONVIVIALITY (ANR-23-CE40-0003) of the
French National Research Agency. We would like to thank Max Fathi and Viktor Stein  for pointing us out, respectively, the HWI inequality and references~\cite{wang2022accelerated,JMLR:v26:23-1288}. The total amount of generative artificial
intelligence tools involved in this work is exactly zero.

\bibliographystyle{plain}
\bibliography{biblio} 

\begin{thebibliography}{10}

\bibitem{albritton2024variational}
Dallas Albritton, Scott Armstrong, Jean-Christophe Mourrat, and Matthew Novack.
\newblock Variational methods for the kinetic fokker--planck equation.
\newblock {\em Analysis \& PDE}, 17(6):1953--2010, 2024.

\bibitem{altschuler2025shifted}
Jason~M Altschuler, Sinho Chewi, and Matthew~S Zhang.
\newblock Shifted composition iv: toward ballistic acceleration for log-concave
  sampling.
\newblock {\em arXiv preprint arXiv:2506.23062}, 2025.

\bibitem{ambrosio2005gradient}
Luigi Ambrosio, Nicola Gigli, and Giuseppe Savar{\'e}.
\newblock {\em Gradient flows: in metric spaces and in the space of probability
  measures}.
\newblock Springer Science \& Business Media, 2005.

\bibitem{andrieu2021hypocoercivity}
Christophe Andrieu, Alain Durmus, Nikolas N{\"u}sken, and Julien Roussel.
\newblock Hypocoercivity of piecewise deterministic markov process-monte carlo.
\newblock {\em The Annals of Applied Probability}, 31(5):2478--2517, 2021.

\bibitem{arbel2019maximum}
Michael Arbel, Anna Korba, Adil Salim, and Arthur Gretton.
\newblock Maximum mean discrepancy gradient flow.
\newblock {\em Advances in neural information processing systems}, 32, 2019.

\bibitem{BakryGentilLedoux}
Dominique Bakry, Ivan Gentil, and Michel Ledoux.
\newblock {\em Analysis and geometry of {M}arkov diffusion operators}, volume
  348 of {\em Grundlehren der Mathematischen Wissenschaften [Fundamental
  Principles of Mathematical Sciences]}.
\newblock Springer, Cham, 2014.

\bibitem{bashiri2020gradient}
Kaveh Bashiri and Anton Bovier.
\newblock Gradient flow approach to local mean-field spin systems.
\newblock {\em Stochastic Processes and their Applications}, 130(3):1461--1514,
  2020.

\bibitem{baudoin2017bakry}
Fabrice Baudoin.
\newblock Bakry--{\'e}mery meet villani.
\newblock {\em Journal of functional analysis}, 273(7):2275--2291, 2017.

\bibitem{Dagallier}
Roland {Bauerschmidt}, Thierry {Bodineau}, and Beno{\^i}t {Dagallier}.
\newblock {A criterion on the free energy for log-Sobolev inequalities in
  mean-field particle systems}.
\newblock {\em arXiv e-prints}, page arXiv:2503.24372, March 2025.

\bibitem{BierkensRoberts}
Joris Bierkens and Gareth Roberts.
\newblock A piecewise deterministic scaling limit of lifted
  {M}etropolis-{H}astings in the {C}urie-{W}eiss model.
\newblock {\em Ann. Appl. Probab.}, 27(2):846--882, 2017.

\bibitem{bouin2026quantitative}
{\'E}meric Bouin and Amic Frouvelle.
\newblock Quantitative stability of constant equilibria in a non-linear
  alignment model of self-propelled particles.
\newblock {\em arXiv preprint arXiv:2604.05927}, 2026.

\bibitem{Brigati}
Giovanni Brigati, Francis Lörler, and Lihan Wang.
\newblock Hypocoercivity meets lifts.
\newblock {\em Kinetic and Related Models}, 20(0):34--55, 2026.

\bibitem{bunne2022proximal}
Charlotte Bunne, Laetitia Papaxanthos, Andreas Krause, and Marco Cuturi.
\newblock Proximal optimal transport modeling of population dynamics.
\newblock In {\em International Conference on Artificial Intelligence and
  Statistics}, pages 6511--6528. PMLR, 2022.

\bibitem{cao2023explicit}
Yu~Cao, Jianfeng Lu, and Lihan Wang.
\newblock On explicit l 2-convergence rate estimate for underdamped langevin
  dynamics.
\newblock {\em Archive for Rational Mechanics and Analysis}, 247(5):90, 2023.

\bibitem{carmona2018probabilistic}
Ren{\'e} Carmona, Fran{\c{c}}ois Delarue, et~al.
\newblock {\em Probabilistic theory of mean field games with applications
  I-II}, volume~3.
\newblock Springer, 2018.

\bibitem{carrillo2019convergence}
Jos{\'e}~A Carrillo, Young-Pil Choi, and Oliver Tse.
\newblock Convergence to equilibrium in wasserstein distance for damped euler
  equations with interaction forces.
\newblock {\em Communications in Mathematical Physics}, 365(1):329--361, 2019.

\bibitem{CMCV}
Jos{\'e}~A. Carrillo, Robert~J. McCann, and C{\'e}dric Villani.
\newblock {Kinetic equilibration rates for granular media and related
  equations: entropy dissipation and mass transportation estimates}.
\newblock {\em Revista Matemática Iberoamericana}, 19(3):971 -- 1018, 2003.

\bibitem{cattiaux2019entropic}
Patrick Cattiaux, Arnaud Guillin, Pierre Monmarch{\'e}, and Chaoen Zhang.
\newblock Entropic multipliers method for langevin diffusion and weighted log
  sobolev inequalities.
\newblock {\em Journal of Functional Analysis}, 277(11):108288, 2019.

\bibitem{Songbo2}
Fan Chen, Yiqing Lin, Zhenjie Ren, and Songbo Wang.
\newblock {Uniform-in-time propagation of chaos for kinetic mean field Langevin
  dynamics}.
\newblock {\em Electronic Journal of Probability}, 29(none):1 -- 43, 2024.

\bibitem{Songbo}
Fan Chen, Zhenjie Ren, and Songbo Wang.
\newblock Uniform-in-time propagation of chaos for mean field langevin
  dynamics.
\newblock {\em arXiv preprint arXiv:2212.03050}, 2022.

\bibitem{JMLR:v26:23-1288}
Shi Chen, Qin Li, Oliver Tse, and Stephen~J. Wright.
\newblock Accelerating optimization over the space of probability measures.
\newblock {\em Journal of Machine Learning Research}, 26(31):1--40, 2025.

\bibitem{chizat}
L{\'e}na{\"i}c Chizat.
\newblock Mean-field langevin dynamics : Exponential convergence and annealing.
\newblock {\em Transactions on Machine Learning Research}, 2022.

\bibitem{10.3150/19-BEJ1178}
Arnak~S. Dalalyan and Lionel Riou-Durand.
\newblock {On sampling from a log-concave density using kinetic Langevin
  diffusions}.
\newblock {\em Bernoulli}, 26(3):1956 -- 1988, 2020.

\bibitem{Pavliotis}
Mat{\'\i}as~G Delgadino, Rishabh~S Gvalani, Grigorios~A Pavliotis, and Scott~A
  Smith.
\newblock Phase transitions, logarithmic sobolev inequalities, and
  uniform-in-time propagation of chaos for weakly interacting diffusions.
\newblock {\em Communications in Mathematical Physics}, pages 1--49, 2023.

\bibitem{10.1214/20-AAP1659}
George Deligiannidis, Daniel Paulin, Alexandre Bouchard-C{\^o}t{\'e}, and
  Arnaud Doucet.
\newblock {Randomized Hamiltonian Monte Carlo as scaling limit of the bouncy
  particle sampler and dimension-free convergence rates}.
\newblock {\em The Annals of Applied Probability}, 31(6):2612 -- 2662, 2021.

\bibitem{diaconis2000analysis}
Persi Diaconis, Susan Holmes, and Radford~M Neal.
\newblock Analysis of a nonreversible markov chain sampler.
\newblock {\em Annals of Applied Probability}, pages 726--752, 2000.

\bibitem{dolbeault2015hypocoercivity}
Jean Dolbeault, Cl{\'e}ment Mouhot, and Christian Schmeiser.
\newblock Hypocoercivity for linear kinetic equations conserving mass.
\newblock {\em Transactions of the American Mathematical Society},
  367(6):3807--3828, 2015.

\bibitem{durmus2024asymptotic}
Alain Durmus and Andreas Eberle.
\newblock Asymptotic bias of inexact markov chain monte carlo methods in high
  dimension.
\newblock {\em The Annals of Applied Probability}, 34(4):3435--3468, 2024.

\bibitem{EberleLift}
Lörler~Francis Eberle, Andreas.
\newblock {Non-reversible lifts of reversible diffusion processes and
  relaxation times}.

\bibitem{even2021continuized}
Mathieu Even, Rapha{\"e}l Berthier, Francis Bach, Nicolas Flammarion, Hadrien
  Hendrikx, Pierre Gaillard, Laurent Massouli{\'e}, and Adrien Taylor.
\newblock Continuized accelerations of deterministic and stochastic gradient
  descents, and of gossip algorithms.
\newblock {\em Advances in Neural Information Processing Systems},
  34:28054--28066, 2021.

\bibitem{fan2026sharp}
Zexi Fan, Bowen Li, and Jianfeng Lu.
\newblock {Sharp hypocoercive convergence estimates for underdamped Langevin
  dynamics via the modified $ L^2$ method}.
\newblock {\em arXiv preprint arXiv:2604.10068}, 2026.

\bibitem{FournierGuillin}
Nicolas Fournier and Arnaud Guillin.
\newblock On the rate of convergence in wasserstein distance of the empirical
  measure.
\newblock {\em Probability theory and related fields}, 162(3-4):707--738, 2015.

\bibitem{gadat2013spectral}
S{\'e}bastien Gadat and Laurent Miclo.
\newblock {Spectral decompositions and $L^2$-operator norms of toy hypocoercive
  semi-groups}.
\newblock {\em Kinetic and related models}, 6(2):317--372, 2013.

\bibitem{geshkovski2025mathematical}
Borjan Geshkovski, Cyril Letrouit, Yury Polyanskiy, and Philippe Rigollet.
\newblock A mathematical perspective on transformers.
\newblock {\em Bulletin of the American Mathematical Society}, 62(3):427--479,
  2025.

\bibitem{M41}
Nicolas {Gouraud}, Pierre {Le Bris}, Adrien {Majka}, and Pierre
  {Monmarch{\'e}}.
\newblock {HMC and Underdamped Langevin United in the Unadjusted Convex Smooth
  Case}.
\newblock {\em SIAM/ASA Journal on Uncertainty Quantification}, 13(1):278--303,
  2025.

\bibitem{GuillinWuZhang}
Arnaud Guillin, Wei Liu, Liming Wu, and Chaoen Zhang.
\newblock {Uniform Poincaré and logarithmic Sobolev inequalities for mean
  field particle systems}.
\newblock {\em The Annals of Applied Probability}, 32(3):1590 -- 1614, 2022.

\bibitem{guillin2021uniform}
Arnaud Guillin and Pierre Monmarch{\'e}.
\newblock Uniform long-time and propagation of chaos estimates for mean field
  kinetic particles in non-convex landscapes.
\newblock {\em Journal of Statistical Physics}, 185:1--20, 2021.

\bibitem{herau2007short}
Fr{\'e}d{\'e}ric H{\'e}rau.
\newblock Short and long time behavior of the fokker--planck equation in a
  confining potential and applications.
\newblock {\em Journal of Functional Analysis}, 244(1):95--118, 2007.

\bibitem{herau2004isotropic}
Fr{\'e}d{\'e}ric H{\'e}rau and Francis Nier.
\newblock Isotropic hypoellipticity and trend to equilibrium for the
  fokker-planck equation with a high-degree potential.
\newblock {\em Archive for Rational Mechanics and Analysis}, 171(2):151--218,
  2004.

\bibitem{Szpruch}
Kaitong Hu, Zhenjie Ren, David Šiška, and Łukasz Szpruch.
\newblock {Mean-field Langevin dynamics and energy landscape of neural
  networks}.
\newblock {\em Annales de l'Institut Henri Poincaré, Probabilités et
  Statistiques}, 57(4):2043 -- 2065, 2021.

\bibitem{karimi2016linear}
Hamed Karimi, Julie Nutini, and Mark Schmidt.
\newblock Linear convergence of gradient and proximal-gradient methods under
  the polyak-{\l}ojasiewicz condition.
\newblock In {\em Joint European conference on machine learning and knowledge
  discovery in databases}, pages 795--811. Springer, 2016.

\bibitem{lambert2022variational}
Marc Lambert, Sinho Chewi, Francis Bach, Silv{\`e}re Bonnabel, and Philippe
  Rigollet.
\newblock Variational inference via wasserstein gradient flows.
\newblock {\em Advances in Neural Information Processing Systems},
  35:14434--14447, 2022.

\bibitem{lelievre2026convergence}
Tony Leli{\`e}vre, Xuyang Lin, and Pierre Monmarch{\'e}.
\newblock Convergence rates for an adaptive biasing potential scheme from a
  wasserstein optimization perspective.
\newblock {\em Nonlinearity}, 39(4):045016, 2026.

\bibitem{lelievre2016partial}
Tony Lelievre and Gabriel Stoltz.
\newblock Partial differential equations and stochastic methods in
  moleculardynamics.
\newblock {\em Acta Numerica}, 25:681--880, 2016.

\bibitem{law1965ensembles}
Stanislaw Lojasiewicz.
\newblock Ensembles semi-analytiques.
\newblock {\em IHES notes}, page 220, 1965.

\bibitem{Lu}
Jianfeng {Lu}.
\newblock {A sharp hypocoercive entropy decay estimate for underdamped Langevin
  dynamics}.
\newblock {\em arXiv e-prints}, page arXiv:2605.01933, May 2026.

\bibitem{lu2022explicit}
Jianfeng Lu and Lihan Wang.
\newblock On explicit l 2-convergence rate estimate for piecewise deterministic
  markov processes in mcmc algorithms.
\newblock {\em The Annals of Applied Probability}, 32(2):1333--1361, 2022.

\bibitem{mei2018mean}
Song Mei, Andrea Montanari, and Phan-Minh Nguyen.
\newblock A mean field view of the landscape of two-layer neural networks.
\newblock {\em Proceedings of the National Academy of Sciences},
  115(33):E7665--E7671, 2018.

\bibitem{menon2026implicit}
Govind Menon, Austin~J Stromme, and Adrien Vacher.
\newblock On the implicit regularization of langevin dynamics with projected
  noise.
\newblock {\em arXiv preprint arXiv:2602.12257}, 2026.

\bibitem{Monmarche2013}
Laurent Miclo and Pierre Monmarch\'{e}.
\newblock \'{E}tude spectrale minutieuse de processus moins ind\'{e}cis que les
  autres.
\newblock In {\em S\'{e}minaire de {P}robabilit\'{e}s {XLV}}, volume 2078 of
  {\em Lecture Notes in Math.}, pages 459--481. Springer, Cham, 2013.

\bibitem{M24}
Pierre Monmarch{\'e}.
\newblock Piecewise deterministic simulated annealing.
\newblock {\em ALEA Lat. Am. J. Probab. Math. Stat.}, 13(1):357--398, 2016.

\bibitem{M18}
Pierre {Monmarch{\'e}}.
\newblock {Generalized $\Gamma$ calculus and application to interacting
  particles on a graph}.
\newblock {\em Potential Analysis}, 50:439--466, 2019.

\bibitem{idealized}
Pierre Monmarch{\'e}.
\newblock {An entropic approach for Hamiltonian Monte Carlo: The idealized
  case}.
\newblock {\em The Annals of Applied Probability}, 34(2):2243 -- 2293, 2024.

\bibitem{MonmarcheToyModelMF}
Pierre {Monmarch{\'e}}.
\newblock {Free energy Wasserstein gradient flow and their particle
  counterparts: toy model, (degenerate) PL inequalities and exit times}.
\newblock {\em arXiv e-prints}, page arXiv:2510.16506, October 2025.

\bibitem{MonmarcheLSI}
Pierre Monmarch{\'e}.
\newblock Uniform log-sobolev inequalities for mean field particles beyond
  flat-convexity.
\newblock {\em Stochastic Processes and their Applications}, 2025.

\bibitem{monmarche2025local}
Pierre Monmarch{\'e} and Julien Reygner.
\newblock Local convergence rates for wasserstein gradient flows and
  mckean-vlasov equations with multiple stationary solutions.
\newblock {\em Probability Theory and Related Fields}, pages 1--59, 2025.

\bibitem{M26}
Pierre {Monmarch{\'e}}, Matthias {Rousset}, and Pierre-Andr{\'e} {Zitt}.
\newblock Exact targeting of gibbs distributions using velocity-jump processes.
\newblock {\em Stochastics and Partial Differential Equations: Analysis and
  Computations}, pages 1--40, 2022.

\bibitem{MONMARCHE20171721}
Pierre Monmarché.
\newblock Long-time behaviour and propagation of chaos for mean field kinetic
  particles.
\newblock {\em Stochastic Processes and their Applications}, 127(6):1721--1737,
  2017.

\bibitem{nesterov1983method}
Yurii Nesterov.
\newblock A method for solving the convex programming problem with convergence
  rate o (1/k2).
\newblock In {\em Dokl akad nauk Sssr}, volume 269, page 543, 1983.

\bibitem{otto2001geometry}
Felix Otto.
\newblock The geometry of dissipative evolution equations: the porous medium
  equation.
\newblock 2001.

\bibitem{PetersdeWith}
E.~A. J.~F. Peters and G.~de~With.
\newblock Rejection-free monte carlo sampling for general potentials.
\newblock {\em Phys. Rev. E 85, 026703}, 2012.

\bibitem{peyre2015entropic}
Gabriel Peyr{\'e}.
\newblock Entropic approximation of wasserstein gradient flows.
\newblock {\em SIAM Journal on Imaging Sciences}, 8(4):2323--2351, 2015.

\bibitem{polyak1964some}
Boris~T Polyak.
\newblock Some methods of speeding up the convergence of iteration methods.
\newblock {\em Ussr computational mathematics and mathematical physics},
  4(5):1--17, 1964.

\bibitem{sandier2004gamma}
Etienne Sandier and Sylvia Serfaty.
\newblock Gamma-convergence of gradient flows with applications to
  ginzburg-landau.
\newblock {\em Communications on Pure and Applied Mathematics: A Journal Issued
  by the Courant Institute of Mathematical Sciences}, 57(12):1627--1672, 2004.

\bibitem{shi2022understanding}
Bin Shi, Simon~S Du, Michael~I Jordan, and Weijie~J Su.
\newblock Understanding the acceleration phenomenon via high-resolution
  differential equations.
\newblock {\em Mathematical Programming}, 195(1):79--148, 2022.

\bibitem{su2016differential}
Weijie Su, Stephen Boyd, and Emmanuel~J Candes.
\newblock A differential equation for modeling nesterov's accelerated gradient
  method: Theory and insights.
\newblock {\em Journal of Machine Learning Research}, 17(153):1--43, 2016.

\bibitem{tse2019quantitative}
Alvin Tsz~Ho Tse.
\newblock Quantitative propagation of chaos of mckean-vlasov equations via the
  master equation.
\newblock 2019.

\bibitem{Villani}
C{\'e}dric Villani.
\newblock Hypocoercivity.
\newblock {\em Mem. Amer. Math. Soc.}, 202(950):iv+141, 2009.

\bibitem{SongboLSI}
Songbo {Wang}.
\newblock {Uniform log-Sobolev inequalities for mean field particles with
  flat-convex energy}.
\newblock {\em arXiv e-prints}, page arXiv:2408.03283, August 2024.

\bibitem{wang2025large}
Songbo Wang.
\newblock Large-scale concentration and relaxation for mean-field langevin
  particle systems.
\newblock {\em arXiv preprint arXiv:2508.16428}, 2025.

\bibitem{wang2022accelerated}
Yifei Wang and Wuchen Li.
\newblock Accelerated information gradient flow.
\newblock {\em Journal of Scientific Computing}, 90(1):11, 2022.

\bibitem{wilson2021lyapunov}
Ashia~C Wilson, Ben Recht, and Michael~I Jordan.
\newblock A lyapunov analysis of accelerated methods in optimization.
\newblock {\em Journal of Machine Learning Research}, 22(113):1--34, 2021.

\end{thebibliography}

\end{document}